\documentclass[a4paper,11pt]{article}
\usepackage{amsmath, amsfonts, amscd, amssymb, amsthm, enumerate, xypic}

\def\bfB{\mathbf{B}}

\newcommand{\Ortho}{\text{O}}
\newcommand{\Sp}{\text{Sp}}
\newcommand{\WS}{\mathcal{N}\mathcal{S}}
\newcommand{\WA}{\mathcal{N}\mathcal{A}}
\renewcommand{\epsilon}{\varepsilon}
\newcommand{\modu}{\operatorname{mod}}

\newcommand{\Mat}{\operatorname{M}}
\newcommand{\Mats}{\operatorname{S}}
\newcommand{\Mata}{\operatorname{A}}

\newcommand{\Ker}{\operatorname{Ker}}

\newcommand{\End}{\operatorname{End}}
\newcommand{\ind}{\operatorname{ind}}

\newcommand{\Vect}{\operatorname{span}}
\newcommand{\im}{\operatorname{Im}}

\newcommand{\tr}{\operatorname{tr}}

\newcommand{\rk}{\operatorname{rk}}

\newcommand{\codim}{\operatorname{codim}}
\renewcommand{\setminus}{\smallsetminus}

\def\F{\mathbb{F}}

\def\N{\mathbb{N}}

\renewcommand{\L}{\mathbb{L}}

\def\calA{\mathcal{A}}

\def\calF{\mathcal{F}}
\def\calG{\mathcal{G}}

\def\calN{\mathcal{N}}

\def\calS{\mathcal{S}}

\def\calU{\mathcal{U}}
\def\calV{\mathcal{V}}

\def\lcro{\mathopen{[\![}}
\def\rcro{\mathclose{]\!]}}

\theoremstyle{definition}
\newtheorem{Def}{Definition}
\newtheorem{Not}[Def]{Notation}

\theoremstyle{plain}
\newtheorem{theo}{Theorem}[section]
\newtheorem{prop}[theo]{Proposition}
\newtheorem{cor}[theo]{Corollary}
\newtheorem{lemma}[theo]{Lemma}
\newtheorem{claim}{Claim}[section]

\theoremstyle{plain}

\theoremstyle{remark}

\title{The structured Gerstenhaber problem (II)}
\author{Cl\'ement de Seguins Pazzis\footnote{Universit\'e de Versailles Saint-Quentin-en-Yvelines, Laboratoire de Math\'ematiques
de Versailles, 45 avenue des Etats-Unis, 78035 Versailles cedex, France, dsp.prof@gmail.com}}

\begin{document}

\thispagestyle{plain}

\maketitle

\begin{abstract}
Let $b$ be a non-degenerate symmetric (respectively, alternating) bilinear form on a finite-dimensional vector space $V$, over a field with characteristic different from $2$. In a previous work \cite{dSPStructured1}, we have determined the maximal possible dimension for a linear subspace of $b$-alternating (respectively, $b$-symmetric) nilpotent endomorphisms of $V$. Here, provided that the cardinality of the underlying field be large enough with respect to the Witt index of $b$, we classify the spaces that have the maximal possible dimension.

Our proof is based on a new sufficient condition for the reducibility of a vector space of nilpotent linear operators. To illustrate
the power of that new technique, we use it to give a short new proof of the classical Gerstenhaber theorem on
large vector spaces of nilpotent matrices (provided, again, that the cardinality of the underlying field be large enough).
\end{abstract}

\vskip 2mm
\noindent
\emph{AMS Classification:} 15A30, 15A63, 15A03.

\vskip 2mm
\noindent
\emph{Keywords:} Symmetric matrices, Skew-symmetric matrices, Nilpotent matrices, Bilinear forms, Dimension, Gerstenhaber theorem.


\section{Introduction}

\subsection{Notation}

Throughout the article and with the exception of Sections \ref{generalsection} and \ref{standardGerstenhabersection}, $\F$ denotes a field with characteristic not $2$. Let $V$ be a finite-dimensional vector space over $\F$, equipped with a
non-degenerate bilinear form $b : V \times V \rightarrow \F$, and assume that $b$ is symmetric or
alternating (i.e. skew-symmetric). We say that $b$ is symplectic when it is non-degenerate and alternating.
Throughout, orthogonality in $V$ is considered with respect to $b$,
except in Section \ref{standardGerstenhabersection} where $b$ will not be mentioned and dual-orthogonality will be considered instead
(of course, no confusion will arise then). A linear subspace $X$ of $V$ is called \textbf{totally singular} for $b$ if $X \subset X^\bot$.
The maximal dimension for such a subspace is called the \textbf{Witt index} of $b$.
Finally, we say that $b$ is non-isotropic whenever its Witt index equals zero, or equivalently
$b(x,x)\neq 0$ for all $x \in V \setminus \{0\}$.

An endomorphism $u$ of $V$ is called $b$-symmetric (respectively, $b$-alternating)
when the bilinear form $(x,y) \in V^2 \mapsto b(x,u(y))$ is symmetric (respectively, alternating).
The set of all $b$-symmetric endomorphisms is denoted by $\calS_b$, while the set of all $b$-alternating
endomorphisms is denoted by $\calA_b$.

A subset of an $\F$-algebra is called \textbf{nilpotent} when all its elements are nilpotent,
and the nilindex of a nilpotent element $x$ of such an algebra is denoted by $\ind(x)$
(it is the least non-negative integer $k$ such that $x^k=0$).

Given non-negative integers $p$ and $n$, we denote by $\Mat_{n,p}(\F)$ the vector space of all $n$ by $p$ matrices
with entries in $\F$, by $\Mat_n(\F)$ the algebra of all square matrices of size $n$ and entries in $\F$
(and by $0_n$ and $I_n$ its zero matrix and its identity matrix, respectively), by $\Mats_n(\F)$ the linear subspace of all $n$ by $n$ symmetric matrices, and
by $\Mata_n(\F)$ the linear subspace of all $n$ by $n$ skew-symmetric matrices.

\subsection{The problem}

The traditional Gerstenhaber problem consists of the following questions:
\begin{itemize}
\item What is the maximal possible dimension for a nilpotent linear subspace of $\End(V)$?
\item What are the spaces with the maximal possible dimension?
\end{itemize}
Remember that a flag $\calF$ of $V$ is a finite increasing sequence
$(F_i)_{0 \leq i \leq p}$ of linear subspaces of $V$, and such a flag is said to be \textbf{stable} under an endomorphism $u$
if $u(F_i) \subset F_i$ for all $i \in \lcro 0,p\rcro$. The flag $(F_0,\dots,F_p)$ is said to be \textbf{complete}
if $F_p=V$ and $\dim F_i=i$ for all $i \in \lcro 0,p\rcro$.
If $\calF$ is complete one denotes by $\calN_\calF$
the set of all endomorphisms $u$ of $V$ such that $u(F_i) \subset F_{i-1}$ for all $i \in \lcro 1,p\rcro$
(which is equivalent to having $u$ nilpotent and stabilize $F_i$ for all $i \in \lcro 0,p\rcro$).
With that in mind, the Gerstenhaber problem has the following classical answer (with no restriction on the characteristic of the underlying field):

\begin{theo}[Gerstenhaber (1958), Serezhkin (1982)]
Let $\calV$ be a nilpotent linear subspace of $\End(V)$. Then, $\dim \calV \leq \dbinom{\dim V}{2}$,
and equality holds if and only if $\calV=\calN_\calF$ for a (unique) complete flag $\calF$ of $V$.
\end{theo}

Gerstenhaber \cite{Gerstenhaber} proved this result under the additional requirement that $\F$ be of cardinality at least $\dim V$;
Serezhkin \cite{Serezhkin} later lifted that assumption.

\vskip 3mm
In this article, we deal with the so-called \emph{structured} Gerstenhaber problem, which is the equivalent of the traditional one
for a space equipped with a non-degenerate symmetric or alternating bilinear form. In short,
we are interested in the nilpotent linear subspaces of $\calS_b$ and $\calA_b$.
Here, the maximal dimension is connected to the Witt index of $b$.

A flag $\calF=(F_0,\dots,F_p)$ of $V$ is called \textbf{partially complete}
whenever $\dim F_i=i$ for all $i \in \lcro 0,p\rcro$.
Such a flag is called \textbf{$b$-singular} whenever $F_p$ is totally singular for $b$. When it is partially complete and $b$-singular,
it is called \textbf{maximal} if $p$ equals the Witt index of $b$.
It was shown in \cite{dSPStructured1} that if $\calF=(F_i)_{0 \leq i \leq \nu}$ is a maximal partially complete $b$-singular flag of $V$,
then the set $\WS_{b,\calF}$ (respectively, the set $\WA_{b,\calF}$) of all \emph{nilpotent} $u \in \calS_b$ (respectively, $u \in \calA_b$)
that stabilize $\calF$ (or, equivalently, $u (F_i) \subset F_{i-1}$ for all $i \in \lcro 1,\nu\rcro$), is a linear subspace with dimension $\nu(n-\nu)$ (respectively, $\nu(n-\nu-1)$).
Moreover, it was proved in \cite{dSPStructured1} that this dimension is optimal for the structured Gerstenhaber problem:

\begin{theo}[See \cite{dSPStructured1} theorem 1.7]\label{majoTheo}
Let $V$ be an $n$-dimensional vector space equipped with a non-degenerate symmetric or alternating bilinear form $b$.
Denote by $\nu$ the Witt index of $b$.
\begin{enumerate}[(a)]
\item Let $\calV$ be a nilpotent linear subspace of $\calS_b$.
Then, $\dim \calV \leq \nu(n-\nu)$.

\item Let $\calV$ be a nilpotent linear subspace of $\calA_b$.
Then, $\dim \calV \leq \nu(n-\nu-1)$.
\end{enumerate}
\end{theo}

This theorem generalizes earlier results of Meshulam and Radwan \cite{MeshulamRadwan} who tackled the case of an algebraically closed field
with characteristic $0$ with a symmetric form $b$. In two specific cases (point (a) when $b$ is alternating, point (b) when $b$
is symmetric), the result can also be seen as a special case of a general result of Draisma, Kraft and Kuttler \cite{DraismaKraftKuttler} when the underlying field is algebraically closed and with characteristic different from $2$.

The present article deals with the problem of characterizing the spaces with maximal dimension.
It turns out that, in contrast with Theorem \ref{majoTheo} -- in the proof of which there is no need to discriminate whether
$b$ is symmetric or alternating -- the search for spaces having the maximal dimension must be split into four subproblems:
spaces of symmetric endomorphisms for a symmetric form, spaces of alternating endomorphisms for an alternating form,
spaces of alternating endomorphisms for a symmetric form, and spaces of symmetric endomorphisms for an alternating form.
A complete solution to the former two subproblems appears in \cite{dSPStructured1}:

\begin{theo}[See \cite{dSPStructured1} theorem 1.8]\label{theosymsym}
Let $V$ be an $n$-dimensional vector space equipped with a non-degenerate symmetric bilinear form $b$, and let
$\calV$ be a nilpotent linear subspace of $\calS_b$ with dimension $\nu(n-\nu)$.
Then, $\calV=\WS_{b,\calF}$ for some maximal partially complete $b$-singular flag $\calF$ of $V$.
\end{theo}

\begin{theo}[See \cite{dSPStructured1} theorem 1.9]
Let $V$ be an $n$-dimensional vector space equipped with a symplectic form $b$,
and let $\calV$ be a nilpotent linear subspace of $\calA_b$ with dimension $\nu(n-\nu-1)$. Then, $\calV=\WA_{b,\calF}$ for some maximal partially complete $b$-singular flag $\calF$ of $V$.
\end{theo}

Theorem \ref{theosymsym} generalizes a recent result of Bukov\v{s}ek and Omladi\v{c} \cite{BukovsekOmladic}, who limited their discussion to complex numbers (their proof can easily be generalized to any quadratically closed field with characteristic not $2$).
In the proof of the above two theorems, the strategy was the same one as in \cite{BukovsekOmladic}:
by a trace orthogonality argument, one proved that $\calV$ is stable under squares; one deduced that $\calV$ is stable under the Jordan product $(u,v) \mapsto uv+vu$; one used the Jacobson triangularization theorem \cite{Jacobson,Radjavi} to deduce that $\calV$ is triangularizable;
the conclusion followed with limited effort.

However, in the two remaining cases, i.e.\ alternating endomorphisms for a symmetric form, and symmetric endomorphisms for
an alternating form, the above strategy fails. Indeed, in the first case (respectively, the second one) the square of an element of $\calA_b$
(respectively, of $\calS_b$) is an element of $\calS_b$ (respectively, of $\calA_b$), and can belong to $\calA_b$ (respectively, to $\calS_b$)
only if it is zero. It is clear though that in general $\WA_{b,\calF}$ (respectively, $\WS_{b,\calF}$) does not consist solely of square-zero elements. It was found however that any nilpotent subspace of $\calA_b$
(respectively, of $\calS_b$) with the maximal possible dimension is stable under cubes
provided that $|\F|>3$ (see lemmas 5.2 and 5.3 of \cite{dSPStructured1}, reproduced as
Lemma \ref{cubeslemma} in Section \ref{stabcubereview} of the present manuscript).
However, a nilpotent subspace that is stable under cubes is not necessarily triangularizable.
A classical example, in matrix terms, is the space of all $3$ by $3$ matrices of the form
$$N(x,y):=\begin{bmatrix}
0 & 0 & x \\
0 & 0 & y \\
-y & x & 0
\end{bmatrix} \quad \text{with $(x,y)\in \F^2$.}$$
One checks that every such matrix has cube zero (and hence this vector space is stable under cubes, and even under any odd power). Yet, it is not triangularizable since it is easily checked that
there is no non-zero vector that lies in the kernel of every matrix of the above type.

In the above example it can be shown that there is no
non-degenerate symmetric or alternating bilinear form $b$ on $\F^3$ for which all the matrices under hand represent $b$-symmetric (or $b$-alternating) endomorphisms in the standard basis. However, the example can be used to construct an example
of this kind. Simply, let $\epsilon \in \{1,-1\}$, and consider the bilinear form $b$ whose
matrix in the standard basis of $\F^6$ is
$$\begin{bmatrix}
0_3 & I_3 \\
\epsilon\,I_3 & 0_3
\end{bmatrix},$$
and the space $\calV$ of all endomorphisms of $\F^6$ with matrix in the standard basis equal to
$$\begin{bmatrix}
N(x,y) & M \\
0_3 & \epsilon'\, N(x,y)^T
\end{bmatrix}$$
for $\epsilon':=\epsilon$ (respectively, $\epsilon':=-\epsilon$), some pair $(x,y)\in \F^2$ and some symmetric (respectively, alternating)
matrix $M \in \Mat_3(\F)$. Then, one checks that $\calV$ is a nilpotent subspace of $\calS_b$
(respectively, of $\calA_b$), and that it is stable under cubes (and, more generally, under any odd power).
Worse still, the Witt index of $b$ is $\nu:=3$, and $\calV$ has dimension exactly $\nu(n-\nu)-1$ (respectively,
$\nu(n-\nu-1)-1$) where $n:=6$, which is just one unit under the critical dimension of Theorem \ref{majoTheo}.
Again, it is not hard to check that $\calV$ is not triangularizable.

Hence, the stability under cubes, although a nice property, is clearly insufficient to
obtain the triangularizability of the spaces with maximal dimension.

\vskip 3mm
The aim of the present article is to give a partial solution in the yet unresolved cases:
assuming that the cardinality of the underlying field is large enough with respect to the Witt index of $b$,
we shall prove that the above characterization still holds.
Here are our results:

\begin{theo}\label{alttheo}
Let $V$ be an $n$-dimensional vector space over $\F$ equipped with a non-degenerate symmetric bilinear form $b$
with Witt index $\nu$, and let
$\calV$ be a nilpotent linear subspace of $\calA_b$.
Assume that $|\F| \geq \min(n,2\nu+1)$ and $\dim \calV=\nu(n-\nu-1)$. Then $\calV=\WA_{b,\calF}$
for some maximal partially complete $b$-singular flag $\calF$ of $V$.
\end{theo}

\begin{theo}\label{symtheo}
Let $V$ be an $n$-dimensional vector space over $\F$ equipped with a symplectic form $b$, and let
$\calV$ be a nilpotent linear subspace of $\calS_b$. Set $\nu:=\frac{n}{2}$ and
assume that $|\F| \geq n$ and $\dim \calV=\nu(n-\nu)$. Then $\calV=\WS_{b,\calF}$
for some maximal partially complete $b$-singular flag $\calF$ of $V$.
\end{theo}

The lower bounds on the cardinality of $\F$ are relevant because they are the lowest ones so that we can ensure
that the nilpotency of a linear subspace of $\calA_b$ (in Theorem \ref{alttheo}) or $\calS_b$ (in Theorem \ref{symtheo})
be preserved in extending the field of scalars; they are connected to the maximal possible nilindex in $\calA_b$ and in $\calS_b$.
See Sections \ref{maxnilindexsection} and \ref{extensionscalarsSection} for details.

At this point, we should note that the above two results are already known in
the case when $\F$ is algebraically closed. Indeed,
in the first case (respectively, the second one) $\calA_b$ (respectively, $\calS_b$) is a Lie subalgebra of $\End(V)$,
and it is isomorphic to the Lie algebra of the reductive algebraic group $\Ortho(b)$ (respectively, $\Sp(b)$).
Moreover, the set of all unipotent orthogonal automorphisms (respectively, of all unipotent symplectic automorphisms)
that stabilize a given maximal partially complete $b$-singular flag $\calF$ is a Borel subgroup
of that algebraic group. A general theorem of Draisma, Kraft and Kuttler
\cite{DraismaKraftKuttler} then yields Theorems \ref{alttheo} and \ref{symtheo} in that case.
However, one must check a tedious condition (named condition (C) in \cite{DraismaKraftKuttler}), and the proof uses deep results from the theory of algebraic groups. In the special case when $b$ has the maximal possible Witt index (that is, $\lfloor n/2\rfloor$),
one can show through an extension of scalars argument that their result yields ours
(this is not entirely straightforward though, we will discuss this issue in Section \ref{extensionscalarsSection}). In particular, Theorem \ref{symtheo}, with its cardinality assumption, is a direct consequence of the general theorem of Draisma et al. However, this still leaves open the problem
of proving Theorem \ref{alttheo} for small Witt indices.

\subsection{Strategy, and structure of the article}

The aim of the present work is to give a (mostly) self-contained elementary proof of Theorems \ref{alttheo} and \ref{symtheo}.
The proof involves a new technique for the study of spaces of nilpotent matrices, which uses elementary algebraic geometry:
the idea is that, given a nilpotent subspace $\calV$ of $\End(V)$ with maximal nilindex $p$,
one looks at the behaviour of the space $\im u^{p-1}$ when $u$ varies in $\calV$. This idea is inspired by
the topic of spaces of bounded rank operators, where similar techniques have been successful.
Combined with standard trace orthogonality techniques, this yields a very interesting sufficient condition for the reducibility of a nilpotent subspace (Lemma \ref{reduclemma}). That condition will be the key to prove the above theorems, and we will
illustrate its power by giving a new proof of the classical Gerstenhaber theorem under the mild assumption that the
underlying field be of cardinality at least the dimension of the vector space $V$ under consideration.

\vskip 3mm
The remainder of the article is organized as follows.

In Section \ref{generalsection}, we recall the classical trace lemma for nilpotent subspaces of operators
(Lemma \ref{tracelemma}), we obtain key results on the spaces $\im u^{p-1}$, when $p$ is the maximal nilindex in
a nilpotent subspace $\calV$ and $u$ varies in $\calV$; we conclude the section by obtaining the key Reducibility Lemma
(Lemma \ref{reduclemma}).

The tools of Section \ref{generalsection} are used in Section \ref{standardGerstenhabersection} to yield
an efficient new proof of Gerstenhaber's theorem for fields with large cardinality: there, the novelty resides in the
analysis of the spaces with the maximal dimension.

In the next two sections, we lay out some groundwork for the structured Gerstenhaber problem.
In Section \ref{groundworksection1}, we collect results on individual $b$-symmetric or $b$-alternating endomorphisms:
in addition to very basic results and considerations of endomorphisms of small rank (the $b$-symmetric tensors and the $b$-alternating tensors),
we study the maximal nilindex in a space of type $\WS_{b,\calF}$ or $\WA_{b,\calF}$, and we finish
with deeper results on the canonical forms for $b$-symmetric or $b$-alternating endomorphisms: those results are needed in the most
technical parts of the proof of Theorem \ref{alttheo}.
In Section \ref{groundworksection2}, we deal more closely with the structured Gerstenhaber problem:
we recall the inductive proof of Theorem \ref{majoTheo} that already appeared in \cite{dSPStructured1}, and
as a by-product of that proof we obtain an important property of spaces with the maximal dimension (the Strong Orthogonality Lemma, see Lemma \ref{strongorthoLemma}), along with results that allow one to perform an inductive proof. From the stability
under cubes that we have already mentioned, we will also derive another stability result
(Corollary \ref{cubescor}) that will turn out to be essential in the proof of Theorem \ref{alttheo}.
We will close Section \ref{groundworksection2} with a discussion on the extension of scalars: it will be shown in particular that, in order to prove Theorems \ref{alttheo} and \ref{symtheo}, it suffices to do so when the underlying field is \emph{infinite}. This will allow us
to use elementary ideas from algebraic geometry.

Having paved the way in Sections \ref{groundworksection1} and \ref{groundworksection2}, we will then
be ready to prove Theorems \ref{alttheo} and \ref{symtheo}. The easier one is by far the latter, so we will start by
tackling it (Section \ref{symproofsection}), and we will finish with the former (Section \ref{altproofsection}).
In both cases, we consider an infinite underlying field, and the proof works by induction on the dimension of the underlying vector space $V$:
the key is to prove the existence of a non-zero isotropic vector $x$ that is annihilated by all the elements of $\calV$.

\section{General results on spaces of nilpotent matrices}\label{generalsection}

Here, $V$ denotes a finite-dimensional vector space over an arbitrary field $\F$
(here, there is no restriction of characteristic).

\subsection{The trace lemma}

One of the main keys for analyzing nilpotent spaces of endomorphisms
is the so-called trace lemma.
It appeared first in \cite{Mathes}, where it was proved for fields with characteristic zero.
The result was extended in \cite{MacD}, and a variation of the proof was given in \cite{dSPStructured1}.

\begin{lemma}[Trace Lemma]\label{tracelemma}
Let $\calV$ be a nilpotent linear subspace of $\End(V)$. Let $u$ and $v$ belong to $\calV$, and let
$k$ be a non-negative integer such that $k < |\F|$. Then, $\tr(u^k v)=0$.
\end{lemma}

\subsection{The generic nilpotency index}

Here, we assume that $V \neq \{0\}$.

\begin{Def}
Let $\calV$ be a nilpotent subspace of $\End(V)$.
The greatest nilindex among the elements of $\calV$ is called the \textbf{generic nilindex} of $\calV$
and denoted by $\ind(\calV)$.
Then, with $p:=\ind(\calV)$, we set
$$\calV^\bullet:=\underset{v \in \calV}{\bigcup} \im v^{p-1} \quad \text{and} \quad
K(\calV):=\Vect (\calV^\bullet).$$
We say that $\calV$ is \textbf{pure} when $\rk u^{p-1} \leq 1$ for all $u \in \calV$.
\end{Def}

In other words, $\calV$ is pure if and only if every element of $\calV$ with nilindex $p$
has exactly one Jordan cell of size $p$.

\begin{Not}
For $x \in V$ and a linear subspace $\calV$ of $\End(V)$, we set
$$\calV x:=\{v(x) \mid v \in \calV\},$$
which is a linear subspace of $V$.
\end{Not}

Let $\calV$ be a nilpotent linear subspace of $\End(V)$.
Given $x \in V \setminus \{0\}$, note that $\calV x$ is linearly disjoint
from $\F x$: indeed, for all $v \in \calV$ we have $v(x) \not\in \F x \setminus \{0\}$ otherwise
$x$ would be an eigenvector of $v$ with a non-zero associated eigenvalue.

Here is an important observation:

\begin{prop}\label{tangentprop}
Let $\calV$ be a nilpotent subspace of $\End(V)$ with generic nilindex $p$.
Assume that $|\F| \geq p$. Let $x \in \calV^\bullet$ and $u \in \calV$ be such that $x \in \im u^{p-1}$.
Then, $\calV x \subset u(K(\calV))$.
\end{prop}

\begin{proof}
Let $v \in \calV$.
We write $x=u^{p-1}(y)$ for some $y \in V$.
Let $\lambda \in \F$. We write
$$(u+\lambda v)^{p-1}=\sum_{k=0}^{p-1} \lambda^k u_k$$
where the operators $u_0,\dots,u_{p-1}$ belong to $\End(V)$, and more precisely
$$u_1=\sum_{k=0}^{p-2} u^k v\,u^{p-2-k}.$$
Then, for all $\lambda \in \F$, we note that $(u+\lambda v)^{p-1}(y)$ belongs to $\calV^\bullet$.

Now, let $\varphi$ be a linear form that vanishes everywhere on $K(\calV)$. Hence,
$$\forall \lambda \in \F, \; \sum_{k=0}^{p-1} \lambda^k \varphi(u_k(y))=0.$$
Since $|\F|\geq p$, this yields in particular $\varphi(u_1(y))=0$. Varying $\varphi$, we obtain
$u_1(y) \in K(\calV)$.

Finally, since $p=\ind(\calV)$ the (vector-valued) polynomial function
$$\lambda \mapsto (u+\lambda v)^p=\sum_{k=0}^{p-1} \lambda^k u'_k$$
vanishes everywhere on $\F$, and again we deduce that $u'_1=0$.
Here, $u'_1=u u_1+v u^{p-1}$, and hence $v(x)=v(u^{p-1}(y))=u(-u_1(y)) \in u(K(\calV))$.
\end{proof}

The next results will be used in the most technical parts of the proof of Theorems \ref{alttheo} and \ref{symtheo}.

\begin{lemma}\label{basicLDL}
Assume that $\F$ is infinite, and let $\calV$ be a pure nilpotent subspace of $\End(V)$.
Then, $\calV^{\bullet}$ cannot be covered by finitely many proper linear subspaces of $K(\calV)$.
\end{lemma}

\begin{proof}
Assume otherwise, and consider a minimal covering $(F_1,\dots,F_N)$ of $\calV^{\bullet}$
by linear hyperplanes of $K(\calV)$. Hence, we have linear forms
$\varphi_1,\dots,\varphi_N$ on $K(\calV)$ such that $H_k=\Ker \varphi_k$ for all $k \in \lcro 1,N\rcro$ and
$$\forall x \in \calV^{\bullet}, \; \prod_{k=1}^N \varphi_k(x)=0.$$
Moreover, as this covering is minimal we can find a vector $x \in \calV^\bullet \setminus \{0\}$ that belongs to none of $H_2,\dots,H_N$.
Choose $u \in \calV$ such that $\im u^{p-1}=\F x$. Let $y \in \calV^\bullet\setminus \{0\}$, and choose $v \in \calV$ such that $\F y=\im v^{p-1}$.
Since $\Ker u^{p-1}$ and $\Ker v^{p-1}$ are proper linear subspaces of $V$, they do not cover $V$ and hence we can find a vector
$z \in V \setminus (\Ker u^{p-1} \cup \Ker v^{p-1})$; then,
$u^{p-1}(z)$ is non-zero and collinear with $x$, and $v^{p-1}(z)$ is non-zero and collinear with $y$.
Set
$$\gamma : \lambda \in \F \mapsto \bigl((1-\lambda)u+\lambda v\bigr)^{p-1}(z) \in \calV^\bullet.$$
The polynomial function
$$\lambda \mapsto \prod_{k=1}^N \varphi_k\bigl(\gamma(\lambda)\bigr)$$
is identically zero. Since $\F$ is infinite, this yields an index
$k \in \lcro 1,p\rcro$ such that $\lambda \mapsto \varphi_k(\gamma(\lambda))$ is identically zero.
However, $\gamma(0)$ is non-zero and collinear with $x$, and $\gamma(1)$ is non-zero and collinear with $y$.
It follows from the first point that we must have $k=1$, and then we deduce from the second one that $y \in H_1$.
Varying $y$ yields $\calV^\bullet \subset H_1$, and hence $K(\calV) \subset H_1$. This is absurd.
\end{proof}

The above lemma will be used through the following (essentially equivalent) form:

\begin{lemma}[Linear Density Lemma]\label{LinearDensityLemma}
Assume that $\F$ is infinite, and let $\calV$ be a pure nilpotent subspace of $\End(V)$.
Let $F_1,\dots,F_N$ be proper linear subspaces of $K(\calV)$. Then, there exists a basis of $K(\calV)$
consisting of vectors of $\calV^\bullet \setminus \underset{k=1}{\overset{N}{\bigcup}} F_k$.
\end{lemma}

\begin{proof}
Assuming otherwise, some linear hyperplane $H$ of $K(\calV)$ would include
$\calV^\bullet \setminus \underset{k=1}{\overset{N}{\bigcup}} F_k$, and hence the proper linear subspaces
$F_1,\dots,F_N,H$ would cover $\calV^\bullet$, contradicting Lemma \ref{basicLDL}.
\end{proof}

\subsection{Application to a reducibility lemma for nilpotent subspaces of operators}

The key to most of the theorems in this article is the following lemma, which relies critically on
Proposition \ref{tangentprop}:

\begin{lemma}[Reducibility Lemma]\label{reduclemma}
Let $V$ be a finite-dimensional vector space over a field $\F$.
Let $\calV$ be a nilpotent linear subspace of $\End(V)$ such that $|\F| \geq \ind(\calV)$.
Let $x \in \calV^\bullet \setminus \{0\}$ be such that
$$\text{(C)} : \quad K(\calV) \subset \F x \oplus \calV x.$$
Then, $\calV x=\{0\}$.
\end{lemma}

In practice, condition (C) is not always easy to obtain, but we will get it with little effort in the case of the traditional Gerstenhaber theorem.

\begin{proof}
Set $p:=\ind(\calV)$ and choose $u \in \calV$ such that $x \in \im u^{p-1}$.

For a vector $y \in V$, the height of $y$ (with respect to $u$) is defined as the supremum
of the integers $k \geq 0$ such that $y \in \im u^k$ (hence, it is $+\infty$ if $y=0$, and at most $p-1$ otherwise).

Since $K(\calV)$ contains the non-zero element $x$, we can choose an element
$y \in K(\calV)$ with minimal height $h<+\infty$.
Condition (C) gives an operator $v \in \calV$ and a scalar $\lambda$ such that $y=v(x)+\lambda x$.
Proposition \ref{tangentprop} then yields a vector $z \in K(\calV)$ such that $v(x)=u(z)$, whence $y=u(z)+\lambda x$.
Since the height of $x$ is at least $p-1$ and the one of $u(z)$ is at least $h+1$, we get that the height
of $y$ would be at least $h+1$ if $h<p-1$, which would be a contradiction in that case.
Hence, $h=p-1$, leading to $K(\calV) \subset \im u^{p-1}$.

We conclude, by using Proposition \ref{tangentprop} once more, that
$$\calV x \subset u(K(\calV))\subset u\bigl(\im u^{p-1}\bigr)=\{0\}.$$
\end{proof}

\section{A new proof of Gerstenhaber's theorem for fields with large cardinality}\label{standardGerstenhabersection}

In this section, we give a new proof of Gerstenhaber's theorem for fields whose cardinality is large
enough with respect to the dimension of the underlying space. The proof takes advantage of the Reducibility Lemma.
Here, $\F$ denotes an arbitrary field (possibly of characteristic $2$).

The proof works by induction. Gerstenhaber's theorem is obviously true for spaces with dimension at most $1$,
so in the rest of the section we fix an integer $n>1$ and we assume that Gerstenhaber's theorem holds
whenever the dimension of the underlying space is less than $n$ and less than or equal to the cardinality of $\F$. We let $V$ be an $n$-dimensional vector space over a field $\F$ with at least $n$ elements,
and $\calV$ be a nilpotent linear subspace of $\End(V)$.

The first part of the proof, which deals with the inequality, is classical.
It can be found in \cite{MacD}: we have to reproduce it because it is a necessary step towards the second part.

\subsection{Proof of the dimension inequality}\label{standardGinequality}

Some preliminary work is required here. Denote by $V^\star$ the dual space of $V$.
Let $f \in V^\star$ and $x \in V$. The endomorphism $y \in V \mapsto f(y)\,x$
is denoted by $f \otimes x$. Its trace equals $f(x)$; it is nilpotent if and only if $f(x)=0$.
Now, fix $x \in V \setminus \{0\}$. We set
$$\calV^x:=\bigl\{f \in V^\star : f \otimes x \in \calV\bigr\}.$$
Through $f\in V^\star \mapsto f \otimes x$, the space $\calV^x$ is isomorphic to the linear subspace of $\calV$
consisting of its elements whose range is included in $\F x$.

\begin{lemma}\label{ortholemmaClassicalGersten}
Let $x \in V \setminus \{0\}$.
The space $\calV^x$ is dual-orthogonal to $\F x \oplus \calV x$.
\end{lemma}

\begin{proof}
Let $f \in \calV^x$ and $u \in \calV$. Then, $(f \otimes x) \circ u =(f \circ u) \otimes x$ has trace $f(u(x))$.
The Trace Lemma applied to $k=1$ yields $f(u(x))=0$. Besides, $f(x)=0$ because $f \otimes x$ is nilpotent,
and the conclusion follows.
\end{proof}

Now, we fix an arbitrary vector $x \in V \setminus \{0\}$.

Denote by $\calU$ the kernel of $u \in \calV \mapsto u(x) \in \calV x$.
Then, every element $u$ of $\calU$ induces a nilpotent endomorphism $\overline{u}$ of $V/\F x$, and
$$\calV \modu x:=\{\overline{u} \mid u \in \calU\}$$
is a nilpotent linear subspace of $\End(V/\F x)$.
By induction, we have
$$\dim (\calV \modu x) \leq \dbinom{n-1}{2}\cdot$$
Denote by $\calU'$ the kernel of $u \in \calU \mapsto \overline{u} \in \calV \modu x$.
Obviously, $\calU'$ is the set of all $u \in \calV$ with range included in $\F x$, whence
$$\calU'=\{f \otimes x \mid f \in \calV^x\} \quad \text{and} \quad \dim \calU'=\dim \calV^x.$$
Applying the rank theorem twice, we obtain
$$\dim \calV=\dim \calV x +\dim \calV^x+\dim (\calV \modu x).$$
By Lemma \ref{ortholemmaClassicalGersten}, $\dim \calV x +\dim \calV^x \leq n-1$, and hence
$$\dim \calV \leq n-1+\dbinom{n-1}{2}=\dbinom{n}{2}.$$

\subsection{Spaces with the maximal dimension}

Assume now that $\dim \calV=\dbinom{n}{2}$. Then, from the above inequalities, we get, for every $x \in V \setminus \{0\}$, the two equalities
$$\dim \calV x +\dim \calV^x=n-1 \quad \text{and} \quad \dim (\calV \modu x)=\dbinom{n-1}{2}.$$
Using Lemma \ref{ortholemmaClassicalGersten} once more, we deduce:

\begin{claim}\label{doubleorthogonalGerstenhaberClassical}
For all $x \in V \setminus \{0\}$, the space $\F x\oplus \calV x$ is the dual orthogonal of $\calV^x$.
\end{claim}

Fix $x \in V \setminus \{0\}$.
Let us assume for a moment that $\calV x=\{0\}$.
Then, $\calU=\calV$. We have just seen that $\calV \modu x$ has dimension $\dbinom{n-1}{2}$, hence by induction
there is a complete flag $(G_0,\dots,G_{n-1})$ of $V/\F x$
that is stable under $\overline{v}$ for all $v \in \calV$.
For all $i \in \lcro 1,n\rcro$, we denote by $F_i$ the inverse image of $G_{i-1}$ under the canonical projection from $V$ to $V/\F x$,
and we deduce that $\calF:=(\{0\},F_1,\dots,F_n)$ is a complete flag of $V$ that is stable under every element of $\calV$.
It ensues that $\calV \subset \calN_\calF$, and since the dimensions are equal we conclude that
$\calV=\calN_\calF$.

Hence, in order to conclude it suffices to exhibit a vector $x \in V \setminus \{0\}$ such that
$\calV x=\{0\}$.
To obtain such a vector, we start by analyzing the generic nilindex of $\calV$, which we denote by $p$.


\begin{claim}
One has $p \in \{n-1,n\}$ and $p \geq 2$.
\end{claim}

\begin{proof}
One has $p \geq 2$ since $\calV \neq \{0\}$.

Choose $x \in V \setminus \{0\}$. Then $\calV \modu x$ is a nilpotent linear subspace of $\End(V/\F x)$
with dimension $\dbinom{n-1}{2}$. By induction it reads $\calN_{\calF'}$ for some complete flag $\calF'$ of $V/\F x$,
and hence it contains an element $v$ such that $v^{n-2} \neq 0$. This yields an operator $u \in \calV$
that annihilates $x$ and whose induced endomorphism of $V/\F x$ equals $v$. Obviously $u^{n-2} \neq 0$, and hence $p>n-2$.
\end{proof}

It follows in particular that every element of $\calV$ with nilindex $p$ has exactly one Jordan cell of size $p$
(and one additional Jordan cell of size $1$ if $p=n-1$).

Now, we fix an element $u$ of $\calV$ with nilindex $p$, and we choose $x \in \im u^{p-1} \setminus \{0\}$.

\begin{claim}\label{powersclaim}
Let $v \in \calV$. Then, $v^k(x) \in \F x\oplus \calV x$ for all $k \in \lcro 0,n-1\rcro$.
\end{claim}

\begin{proof}
Let $k \in \lcro 0,n-1\rcro$. For all $f \in \calV^x$, the Trace Lemma yields
$$0=\tr\bigl((f \otimes x)\circ v^k\bigr)=\tr\bigl((f \circ v^k) \otimes x\bigr)=f(v^k(x)).$$
From Claim \ref{doubleorthogonalGerstenhaberClassical}, we deduce that $v^k(x) \in \F x\oplus \calV x$.
\end{proof}

\begin{claim}\label{alternativeClaim}
For all $v \in \calV$, one has $v(x)=0$ or $\im v^{p-1} \subset \F x\oplus \calV x$.
\end{claim}

\begin{proof}
Let $v \in \calV$ be such that $v(x) \neq 0$.
If the nilindex of $v$ is not $p$, then $\im v^{p-1}=\{0\} \subset \F x\oplus \calV x$.
Assume otherwise. Then, $\im v^{p-1}=\im v \cap \Ker v$ and $\im v^{p-1}$ has dimension $1$: indeed, either $v$ has exactly one Jordan cell, of size $n$,
or it has one Jordan cell of size $p$ and one Jordan cell of size $1$.
Denoting by $k$ the maximal positive integer such that $v^k(x) \neq 0$, we deduce that $v^k(x)$ spans $\im v^{p-1}$.
We conclude by Claim \ref{powersclaim} that $\im v^{p-1} \subset \F x\oplus \calV x$.
\end{proof}

\begin{claim}
One has $\calV x=\{0\}$ or $K(\calV) \subset \F x\oplus \calV x$.
\end{claim}

\begin{proof}
Assume that $\calV x \neq \{0\}$ and choose $v_0 \in \calV$ such that $v_0(x) \neq 0$.
We choose a linear form $\psi$ on $V$ such that $\psi(v_0(x)) \neq 0$.
Let $v \in \calV$ and $z \in V$.
For all $(\lambda,\mu) \in \F^2$, the vector
$$\gamma(\lambda,\mu):=\bigl(\lambda v_0+\mu v)^{p-1}(z)$$
belongs to $K(\calV)$.
For all $(\lambda,\mu) \in \F^2$ such that $(\lambda v_0+\mu v)(x) \neq 0$,
we know from Claim \ref{alternativeClaim} that
$$\bigl(\lambda v_0+\mu v\bigr)^{p-1}(z) \in \F x\oplus \calV x.$$
Let $\varphi$ be a linear form that vanishes everywhere on $\F x\oplus \calV x$.
Then, the $p$-homogeneous polynomial function
$$(\lambda,\mu) \mapsto \psi\bigl((\lambda v_0+\mu v)(x)\bigr)\,\varphi\bigl((\lambda v_0+\mu v)^{p-1}(z)\bigr)$$
vanishes everywhere on $\F^2$, and since $|\F| \geq p$ we successively deduce that this function is identically zero,
and that one of the polynomial functions $(\lambda,\mu) \mapsto \psi\bigl((\lambda v_0+\mu v)(x)\bigr)$ and
$(\lambda,\mu)  \mapsto \varphi\bigl((\lambda v_0+\mu v)^{p-1}(z)\bigr)$ is identically zero.
The first one is not identically zero since $\psi(v_0(x)) \neq 0$, and hence the second one is identically zero.
Taking $(\lambda,\mu)=(0,1)$ and varying $\varphi$ yields $v^{p-1}(z) \in \F x \oplus \calV x$.
Varying $v$ and $z$ finally yields $K(\calV) \subset \F x\oplus \calV x$.
\end{proof}

If $K(\calV) \subset \F x \oplus \calV x$ then we deduce from the Reducibility Lemma
that $\calV x=\{0\}$. Otherwise we directly have $\calV x=\{0\}$.

From there, the conclusion follows, as was explained earlier.

\section{General results on $b$-symmetric and $b$-alternating endomorphisms}\label{groundworksection1}

\subsection{Basic results}

We start with three basic results that reproduce, respectively, lemmas 1.4, 1.5 and 2.4 from \cite{dSPStructured1}.

\begin{lemma}\label{stableortholemma}
Let $b$ be a non-degenerate symmetric or alternating bilinear form on $V$, and
let $u \in \calS_b \cup \calA_b$.
Then:
\begin{enumerate}[(a)]
\item For every linear subspace $W$ of $V$ that is stable under $u$, the subspace
$W^\bot$ is stable under $u$.
\item One has $\Ker u=(\im u)^\bot$.
\end{enumerate}
\end{lemma}

Note that, for all $u \in \calS_b \cup \calA_b$, every power of $u$ belongs to $\calS_b \cup \calA_b$,
and in particular $\im u^k=(\Ker u^k)^\bot$ for every integer $k \geq 0$.

\begin{lemma}\label{nonisotropicLemma}
Let $b$ be a non-isotropic symmetric bilinear form on $V$, and
let $u \in \calS_b \cup \calA_b$ be nilpotent.
Then, $u=0$.
\end{lemma}

\begin{lemma}\label{flaglemma}
Let $b$ be a non-degenerate symmetric or alternating bilinear form on $V$, and
let $u \in \calS_b \cup \calA_b$ be nilpotent.
Then, there exists a maximal partially complete $b$-singular flag $\calF$ of $V$
that is stable under $u$.
\end{lemma}

As a consequence of this last lemma, we find:

\begin{lemma}\label{SETIlemma}
Let $b$ be a non-degenerate symmetric or alternating bilinear form on $V$, and
let $u \in \calS_b \cup \calA_b$ be nilpotent.
Denote by $\nu$ the Witt index of $b$. Then, there exists
a totally $b$-singular subspace $F$ of $V$ with dimension $\nu$
that is stable under $u$. Moreover, $u$ maps $F^\bot$ into $F$ and
$\rk u \leq 2\nu$.
\end{lemma}

\begin{proof}[Proof of Lemma \ref{SETIlemma}]
By Lemma \ref{flaglemma}, we can find a totally $b$-singular subspace $F$ of $V$ with dimension $\nu$
that is stable under $u$. Then $F \subset F^\bot$ and $F^\bot$ is stable under $u$.
The bilinear form $b$ induces a non-degenerate bilinear form $\overline{b}$ on $F^\bot/F$
that is still $b$-symmetric or $b$-alternating. Moreover, since $F$ is maximal among the totally $b$-singular subspaces of
$V$, the form $\overline{b}$ is non-isotropic, and in particular it is symmetric
(if it is alternating then $F^\bot/F=\{0\}$ and it is also symmetric). Moreover, $u$ also induces an
endomorphism $\overline{u}$ of $F^\bot/F$, which is $\overline{b}$-symmetric or $\overline{b}$-alternating.
By Lemma \ref{nonisotropicLemma}, we deduce that $\overline{u}=0$, which means that $u$ maps $F^\bot$ into $F$.

It follows that
$$\rk u \leq \codim_V(F^\bot)+\dim u(F^\bot) \leq 2 \dim F=2\nu.$$
\end{proof}

\begin{cor}\label{MaxRankNil}
Let $b$ be a non-degenerate symmetric or alternating bilinear form on $V$ (with $n:=\dim V>0$), and
let $u \in \calS_b \cup \calA_b$ be nilpotent.
Denote by $\nu$ the Witt index of $b$. Then, $\rk u \leq \min(2\nu,n-1)$ and
$\ind(u) \leq 2\nu+1$. In addition $\rk u \leq n-2$ and $\ind(u) \leq n-1$ if $n=2\nu$ and $u$ is $b$-alternating.
\end{cor}

\begin{proof}
We have shown in the previous lemma that $\rk u \leq 2\nu$.
Obviously $\rk u \leq n-1$ because $u$ is non-injective.
Assume finally that $u$ is $b$-alternating and $n=2\nu$.
Since $b$ is non-degenerate, the rank of $u$ equals the one of the alternating bilinear form
$(x,y) \mapsto b\bigl(x,u(y)\bigr)$, and hence it is even. Since $\rk u<2\nu$ we deduce that $\rk u \leq 2\nu-2$.

The statements on the nilindex are then straightforward consequences of the observation that $\rk u \geq \ind(u)-1$.
\end{proof}

\subsection{Symmetric and alternating $b$-tensors}

Here, $b$ denotes an arbitrary non-degenerate symmetric or alternating bilinear form on
a finite-dimensional vector space $V$.
Given vectors $x$ and $y$ of $V$, we denote by
$$x \otimes_b y : z\in V \mapsto b(y,z)\,x+b(x,z)\,y$$
the \textbf{$b$-symmetric tensor} of $x$ and $y$, and we note that $x \otimes_b y$ belongs to
$\calS_b$; likewise,
$$x \wedge_b y : z\in V \mapsto b(y,z)\,x-b(x,z)\,y,$$
which  belongs to $\calA_b$, is called the \textbf{$b$-alternating tensor} of $x$ and $y$.

Given a non-zero vector $x \in V$, the mapping $y \in V \mapsto x \otimes_b y$
is linear and injective: the direct image of a linear subspace $L$ of $V$ under it is denoted by $x \otimes_b L$.

Given a non-zero vector $x \in V$, the mapping $y \in V \mapsto x \wedge_b y$ is linear and its kernel equals $\F x$:
the direct image of a linear subspace $L$ of $V$ under it is denoted by $x \wedge_b L$.

Now, let $x,y$ be $b$-orthogonal vectors of $V$ such that $b(x,x)=0$.
It is easily seen that $x \otimes_b y$ and $x \wedge_b y$ vanish at $x$ and map $\{x\}^\bot$ into $\F x$.
Since $\Vect(x,y) \subset \{x\}^\bot$, it easily follows that both maps are nilpotent (with nilindex at most $3$).
There is a converse statement, and the following result, which was proved in \cite{dSPStructured1}
(see proposition 3.1 there), sums everything up:

\begin{prop}\label{caractensors}
Let $x$ be a non-zero $b$-isotropic vector of $V$.
Let $u \in \calS_b$ (respectively, $u \in \calA_b$). The following conditions are equivalent:
\begin{enumerate}[(i)]
\item The endomorphism $u$ is nilpotent, vanishes at $x$, and maps $\{x\}^\bot$ into $\F x$.
\item There exists $y \in \{x\}^\bot$ such that $u=x \otimes_b y$ (respectively,
$u=x\wedge_b y$).
\end{enumerate}
\end{prop}

Moreover, the following result is an easy consequence of the Trace Lemma:

\begin{prop}[Orthogonality Lemma for Tensors]\label{ortholemmatensors}
Let $x$ be a non-zero isotropic vector of $V$.
Let $\calV$ be a nilpotent subspace of $\calS_b$ (respectively, of $\calA_b$).
Let $v \in \calV$, and let $y \in \{x\}^\bot$ be such that $x \otimes_b y \in \calV$
(respectively, $x \wedge_b y \in \calV$). Then, for every odd positive integer $k<|\F|$, one has
$$b\bigl(y,v^k(x)\bigr)=0.$$
\end{prop}

\begin{proof}
Let $k$ be an odd positive integer such that $|\F|>k$.
Remembering the notation for tensors from Section \ref{standardGinequality}, we note that, with $\varphi : z \mapsto b(x,z)$ and $\psi : z \mapsto b(y,z)$, we have
$$(x \otimes_b y) \circ v^k=(\psi \circ v^k) \otimes x+(\varphi \circ v^k) \otimes y$$
and
$$(x \wedge_b y) \circ v^k=(\psi \circ v^k) \otimes x-(\varphi \circ v^k) \otimes y$$
and hence, noting that $v^k$ belongs to $\calS_b$ (respectively, to $\calA_b$) because $k$ is odd,
$$\tr\bigl((x \otimes_b y) \circ v^k\bigr)=(\psi \circ v^k)(x)+(\varphi \circ v^k)(y)
=b\bigl(y,v^k(x)\bigr)+b\bigl(x,v^k(y)\bigr)=2\,b\bigl(y,v^k(x)\bigr)$$
(respectively,
$$\tr\bigl((x \wedge_b y) \circ v^k\bigr)=(\psi \circ v^k)(x)-(\varphi \circ v^k)(y)
=b\bigl(y,v^k(x)\bigr)-b\bigl(x,v^k(y)\bigr)=2\,b\bigl(y,v^k(x)\bigr)).$$
From there, the result follows from the Trace Lemma (Lemma \ref{tracelemma}).
\end{proof}


\subsection{The maximal nilindex in $\WS_{b,\calF}$ and $\WA_{b,\calF}$}\label{maxnilindexsection}

Here, given a non-degenerate symmetric or alternating bilinear form $b$ on $V$, with Witt index $\nu$, and
given a maximal partially complete $b$-singular flag $\calF$ of $V$, we investigate the maximal
possible nilindex for an element of $\WA_{b,\calF}$ or of $\WS_{b,\calF}$.
Those results can be retrieved from the (complicated) classification of pairs consisting of a symmetric bilinear form
and an alternating bilinear form under congruence (see e.g.\ \cite{Sergeichuk}), but we will give elementary proofs
for readers that are not well-versed in that theory.

\begin{lemma}\label{genericnilindexalt}
Let $b$ be a non-degenerate symmetric bilinear form on $V$, with Witt index $\nu$. Set $n:=\dim V$.
Let $\calF$ be a maximal partially complete $b$-singular flag of $V$.
If $n \neq 2\nu$, the generic nilindex of $\WA_{b,\calF}$ is $2\nu+1$.
If $n=2\nu$, the generic nilindex of $\WA_{b,\calF}$ is $n-1$.
\end{lemma}

\begin{proof}
By Corollary \ref{MaxRankNil}, it suffices to exhibit an element of $\WA_{b,\calF}$ with
nilindex at least $2\nu+1$ if $n \neq 2\nu$, and at least $n-1$ otherwise.

In any case, we write $\calF=(F_i)_{0 \leq i \leq \nu}$, and set $F:=F_\nu$ and $p:=n-2\nu$;
we take a Witt decomposition
$V=F \oplus G \oplus H$ in which $H$ is totally $b$-singular with dimension $\nu$ and $G=(F \oplus H)^\bot$.
We choose a basis $(f_1,\dots,f_\nu)$ of $F_\nu$ that is adapted to $\calF$ (so that $F_i=\Vect(f_1,\dots,f_i)$ for all $i \in \lcro 1,\nu\rcro$)
and then we obtain a basis $(h_1,\dots,h_\nu)$ of $H$ such that
$b(f_i,h_j)=\delta_{i,j}$ for all $i,j$ in $\lcro 1,\nu\rcro$. We choose an orthogonal basis $(g_1,\dots,g_p)$ of $G$.
The family $\bfB:=(e_1,\dots,e_\nu,g_1,\dots,g_p,h_1,\dots,h_\nu)$ is a basis of $V$ in which the matrix of $b$
reads
$$S:=\begin{bmatrix}
0_\nu & [0]_{\nu \times p} & I_\nu \\
[0]_{p \times \nu} & P & [0]_{p \times \nu} \\
I_\nu & [0]_{\nu \times p} & 0_\nu \\
\end{bmatrix}$$
where $P:=\bigl(b(g_i,g_j)\bigr)_{1 \leq i,j \leq p}$.
We denote by $J:=(\delta_{i+1,j})_{1\leq i,j \leq \nu} \in \Mat_\nu(\F)$ the nilpotent Jordan cell of size $\nu$.

Assume first that $n=2\nu$, and define $K=(k_{i,j}) \in \Mata_\nu(\F)$ by $k_{i,j}=0$ whenever $\{i,j\} \neq \{\nu-1,\nu\}$,
and $k_{\nu,\nu-1}=-k_{\nu-1,\nu}=1$. Setting
$$M:=\begin{bmatrix}
J & K \\
0_\nu & -J^T \\
\end{bmatrix} \in \Mat_n(\F),$$
we see that $SM$ is alternating and $M$ is nilpotent. Hence the endomorphism $u$ represented by $M$ in $\bfB$
is $b$-alternating and nilpotent. Judging from the first $\nu$ columns of $M$, this endomorphism stabilizes the flag $\calF$.
Finally, we see that $u^{\nu-1}(h_1)=(-1)^{\nu-1}(-e_\nu+h_\nu)$,
and then $u^\nu(h_1)=(-1)^{\nu} 2 e_{\nu-1}$, $u^{2\nu-2}(h_1)=(-1)^{\nu}2 e_1$,
whence $u^{2\nu-2} \neq 0$. It follows that $\ind(u) > 2\nu-2$, i.e.\ $\ind(u) \geq n-1$.

Assume now that $n>2\nu$. We define $C:=(c_{i,j}) \in \Mat_{p,\nu}(\F)$
by $c_{i,j}=0$ whenever $(i,j) \neq (p,\nu)$, and $c_{p,\nu} =1$.
We also set
$$N:=\begin{bmatrix}
J & -(PC)^T & 0_\nu \\
[0]_{p \times \nu} & 0_p & C \\
0_\nu & [0]_{\nu \times p} & -J^T \\
\end{bmatrix} \in \Mat_n(\F),$$
and we see that $SN$ is alternating and $N$ is nilpotent. Hence the endomorphism $u$ represented by $N$ in $\bfB$
is $b$-alternating and nilpotent. Once more, we see that $u$ stabilizes $\calF$.
This time around $u^{\nu}(h_1)=(-1)^{\nu-1} g_p$, then $u^{\nu+1}(h_1)=(-1)^\nu \,b(g_p,g_p)\,e_\nu$, and then
$u^{2\nu}(h_1)=(-1)^\nu \,b(g_p,g_p)\,e_1$, to the effect that $\ind(u)>2\nu$.
\end{proof}

Now, we turn to symmetric endomorphisms for a symplectic form:

\begin{lemma}\label{genericnilindexsym}
Let $b$ be a symplectic form on an $n$-dimensional vector space $V$. Set $n:=\dim V$.
Let $\calF$ be a maximal partially complete $b$-singular flag of $V$.

Then, the generic nilindex of $\WS_{b,\calF}$ equals $n$.
\end{lemma}

\begin{proof}
Set $\nu:=n/2$.
We choose a hyperbolic basis $\bfB=(e_1,\dots,e_\nu,f_1,\dots,f_\nu)$ of $V$
in which $(e_1,\dots,e_\nu)$ is adapted to $\calF$, so that
the matrix of $b$ in $\bfB$ reads
$$A:=\begin{bmatrix}
0_\nu & I_\nu \\
-I_\nu & 0_\nu
\end{bmatrix}.$$
Here, we consider once more the $\nu$ by $\nu$ nilpotent Jordan matrix $J$, and this time around we
consider the matrix $D \in \Mats_\nu(\F)$ with exactly one non-zero entry, located at the $(\nu,\nu)$ spot and which equals $1$.
One defines $u$ as the endomorphism of $V$ whose matrix in the basis $\bfB$ equals
$$M:=\begin{bmatrix}
J & D \\
0_\nu & -J^T
\end{bmatrix}.$$
Seeing that $AM$ is symmetric, we obtain that $u$ is $b$-symmetric. Once again, $u$ is nilpotent and stabilizes $\calF$, and this time around we see that
$u^{\nu}(f_1)=(-1)^{\nu-1} e_\nu$ and hence $u^{2\nu-1}(f_1)=(-1)^{\nu-1} e_1$, leading to $\ind(u) \geq 2\nu$.
Conversely, every nilpotent element of $\End(V)$ has nilindex at most $n=2\nu$.
\end{proof}

Combining the above results with Lemma \ref{flaglemma}, we find:

\begin{cor}\label{nilindexCor}
Let $b$ be a non-degenerate symmetric or alternating bilinear form on an $n$-dimensional vector space, with Witt index $\nu$.
\begin{enumerate}[(a)]
\item If $b$ is alternating then the generic nilindex of $\calS_b$ equals $n$.
\item If $b$ is symmetric and $n \neq 2\nu$, then the generic nilindex of $\calA_b$ equals $2\nu+1$.
\item If $b$ is symmetric and $n=2\nu$, then the generic nilindex of $\calA_b$ equals $n-1$.
\end{enumerate}
\end{cor}

\subsection{On the indecomposable nilpotent elements in $\calS_b$ and $\calA_b$}

An endomorphism $u$ of a finite-dimensional vector space $U$ is called a (nilpotent) \textbf{Jordan
cell} of size $p$ if $\dim U=p$ and if there exists a basis $(e_1,\dots,e_p)$ of $U$ such that
$u(e_i)=e_{i-1}$ for all $i \in \lcro 2,p\rcro$, and $u(e_1)=0$. Such a basis is then called a \textbf{Jordan basis} of $u$.

Here, we assume that $b$ is non-degenerate.
Let $u$ belong to $\calS_b$ (respectively, to $\calA_b$). We say that $u$ is \textbf{indecomposable}
when $V \neq \{0\}$ and there is no decomposition of $V$ into a $b$-orthogonal direct sum $V=V_1 \overset{\bot}{\oplus} V_2$
in which $V_1$ and $V_2$ are non-zero linear subspaces that are stable under $u$.
Thanks to the classification theory for pairs of bilinear forms in which one form is symmetric and the other one
is alternating (see, e.g. \cite{Sergeichuk}), we have the following information on indecomposable nilpotent endomorphisms:

\begin{prop}\label{Jordansymalt}
Assume that $b$ is symmetric.
Let $u \in \calA_b$ be nilpotent and indecomposable.
Then, $u$ is either a Jordan cell of odd size or the direct sum of two Jordan cells of the same even size.
Moreover, the Witt index of $b$ equals $\bigl\lfloor \frac{n}{2}\bigr\rfloor$, where $n:=\dim V$.
\end{prop}

We give an elementary proof of this result. In the prospect of that proof, remember that a square matrix $M=(m_{i,j})_{1 \leq i,j \leq n}$ is called \textbf{anti-lower-triangular} whenever $m_{i,j}=0$ for all
$(i,j)\in \lcro 1,n\rcro^2$ with $i+j<n+1$; it is called \textbf{anti-diagonal} whenever
$m_{i,j}=0$ for all $(i,j)\in \lcro 1,n\rcro^2$ with $i+j\neq n+1$. The anti-diagonal entries of $M$ are $m_{1,n}, m_{2,n-1},\dots,m_{n,1}$. An anti-lower-triangular matrix is invertible if and only if all its anti-diagonal entries are non-zero.

\begin{proof}
We consider a splitting $V=V_1 \oplus \cdots \oplus V_N$ in which every subspace $V_i$ is stable under $u$
and the resulting endomorphism of $V_i$ is a Jordan cell. We denote by $q$ the minimal dimension among the $V_i$'s, and we set
$U:=\Ker u^q$.

Assume first that $q$ is odd, and write it $q=2p+1$.
Assume that $b(x,u^{2p}(x))=0$ for every $x \in U$. Then, the quadratic form $x \mapsto b(u^p(x),u^p(x))$ vanishes
everywhere on $U$ and hence $u^p(U)$ -- which equals $\Ker u^{p+1}$ because of the definition of $q$ --
is totally $b$-singular. Yet its orthogonal complement under $b$ is $\im u^{p+1}$. In particular $\Ker u^{p+1} \subset \im u^{p+1}$,
which contradicts the existence of a Jordan cell of size $q$ for $u$.

This yields some $x \in U$ such that $b(x,u^{2p}(x))\neq 0$. Note that $u^{2p+1}(x)=0$.
Let $(k,l)\in \N^2$. Then, $b(u^k(x),u^l(x))=(-1)^k b(x,u^{k+l}(x))=0$ if $k+l>2p$, whereas
$b(u^k(x),u^l(x)) \neq 0$ if $k+l=2p$. It follows that the matrix $\bigl(b(u^{2p+1-k}(x),u^{2p+1-l}(x))\bigr)_{1 \leq k,l \leq 2p+1}$
is anti-lower-triangular with all its anti-diagonal entries non-zero.
Therefore, the subspace $W:=\Vect(u^k(x))_{0 \leq k \leq 2p}$ has dimension $2p+1$ and is $b$-regular,
and the endomorphism induced by $u$ on $W$ is a Jordan cell.
Since $W$ is $b$-regular, $V=W \oplus W^\bot$. As $u$ is indecomposable and $W^\bot$ is stable under $u$, we conclude that $W=V$ and that $u$ is a Jordan cell of odd size. Moreover, we see that $\Vect(u^k(x))_{p+1 \leq k \leq 2p}$ is totally $b$-singular with dimension $p=\frac{n-1}{2}=\bigl\lfloor \frac{n}{2}\bigr\rfloor$,
and hence the Witt index of $b$ equals $\bigl\lfloor \frac{n}{2}\bigr\rfloor$.

Now, we consider the case when $q$ is even, and we write it $q=2p+2$ for some integer $p \geq 0$.
This time around, we set $c : (x,y) \in V^2 \mapsto b(x,u(y))$, which is alternating.
If $c(x,u^{2p}(y))=0$ for all $(x,y)\in U^2$, then $u^p(U)=\Ker u^{p+2}$
would be totally singular for $c$; in other words $\Ker u^{p+2}$ would be $b$-orthogonal to $u(u^p(U))=\Ker u^{p+1}$.
Yet, the $b$-orthogonal complement of $\Ker u^{p+1}$ is $\im u^{p+1}$, and hence we would have $\Ker u^{p+2} \subset \im u^{p+1}$,
contradicting the existence of a Jordan cell of size $2p+2$ for $u$.
This yields two vectors $x$ and $y$ of $\Ker u^q$ such that
$b(x,u^{q-1}(y)) \neq 0$. Now, we set
$$W_x:=\Vect(u^k(x))_{0 \leq k \leq q-1} \quad \text{and} \quad W_y:=\Vect(u^k(y))_{0 \leq k \leq q-1.}$$
We claim that $W_x$ and $W_y$ have dimension $q$ and are linearly disjoint, and that $W_x \oplus W_y$ is $b$-regular.
To back this up, we consider the family
$$(e_1,\dots,e_{2q}):=\bigl(u^{q-1}(y),u^{q-1}(x),\dots,u(y),u(x),y,x\bigr).$$
For all non-negative integers $k,l$ such that $k+l>q-1$, we have $b(u^k(x),u^l(y))=(-1)^k b(x,u^{k+l}(y))=0$,
and likewise $b(u^k(x),u^l(x))=0=b(u^k(y),u^l(y))$. Since $c$ is alternating and $q-1$ is odd, $b(x,u^{q-1}(x))=0=b(y,u^{q-1}(y))$,
whence $b(u^k(x),u^l(x))=0=b(u^k(y),u^l(y))$ for all $(k,l)\in \lcro 0,q-1\rcro^2$ such that $k+l=q-1$. Finally,
$b(u^k(x),u^l(y))=(-1)^{k} b(x,u^{q-1}(y)) \neq 0$ for all non-negative integers $k,l$ such that $k+l=q-1$. It follows that the matrix
$\bigl(b(e_i,e_j)\bigr)_{1 \leq i,j \leq 2q}$ is anti-lower-triangular with all its anti-diagonal entries non-zero.
This shows that $e_1,\dots,e_{2q}$ are linearly independent and that their span is $b$-regular, as claimed.
Since $u$ is indecomposable, we deduce that $V=W_x\oplus W_y$, and hence $u$ is the direct sum of two Jordan cells of size $q$.
Here, we directly have $\Ker u^{p+1}=\im u^{p+1}$, and hence $\Ker u^{p+1}$ is totally $b$-singular with dimension $\frac{n}{2}=q$.
We conclude that the Witt index of $b$ equals $\frac{n}{2}\cdot$
\end{proof}

Here is the corresponding result for nilpotent symmetric endomorphisms for a symplectic form.
We will not need it, so we simply state it and leave the proof (which an easy adaptation of the previous one) to the reader:

\begin{prop}\label{Jordanaltsym}
Assume that $b$ is alternating.
Let $u \in \calS_b$ be nilpotent and indecomposable.
Then, $u$ is either a Jordan cell of even size or the direct sum of two Jordan cells of the same odd size.
\end{prop}

We finish with a very simple result:

\begin{lemma}\label{JordancellLemma}
Let $u \in \calS_b \cup \calA_b$ be a Jordan cell, and $(e_1,\dots,e_n)$ be a Jordan basis of it, with $n$ odd.
Write $n=2p+1$. Then, the radical of the induced bilinear form on $\Vect(e_1,\dots,e_{p+1})$ is $\Vect(e_1,\dots,e_p)$.
\end{lemma}

\begin{proof}
We have $\Ker u^p=\Vect(e_1,\dots,e_p)=\im u^{p+1}$ and $\im u^p=\Vect(e_1,\dots,e_{p+1})$, whence the said radical equals
$$\im u^p \cap (\im u^p)^\bot=\im u^p \cap \Ker u^p=\Vect(e_1,\dots,e_p).$$
\end{proof}

\section{General results on the structured Gerstenhaber problem}\label{groundworksection2}

In this section, we recall some elements of the study of the structured Gerstenhaber problem from \cite{dSPStructured1}, and
we consider the problem of extending scalars.

There are two items we have to review.
First (Section \ref{proofreview}), we need to go back to the inductive proof of Theorem \ref{majoTheo}, and for two reasons:
\begin{itemize}
\item It will explain how Theorems \ref{alttheo} and \ref{symtheo} can be proved by induction.
\item An important partial consequence on the structure of spaces of maximal
dimension, called the Strong Orthogonality Lemma (Lemma \ref{strongorthoLemma}), can be drawn from this proof.
\end{itemize}
In Section \ref{stabcubereview}, we recall a stability property for spaces with maximal dimension that was mentioned in the
introduction, and we draw an important consequence of it.

Finally, in Section \ref{extensionscalarsSection}, we explain how, in Theorems \ref{alttheo} and \ref{symtheo},
one can reduce the situation to the one where the underlying field is infinite.

\subsection{A review of the proof of Theorem \ref{majoTheo}, and some consequences on spaces with the maximal dimension}\label{proofreview}

Here, we recall the inductive proof of Theorem \ref{majoTheo} (see section 3 of \cite{dSPStructured1}).

Let $\calV$ be a nilpotent subspace of $\calS_b$ (respectively, of $\calA_b$).
If the Witt index $\nu$ of $b$ is zero, then Lemma \ref{nonisotropicLemma} shows that $\calV=\{0\}$.
Assume now that $\nu>0$ and let $x \in V$ be an arbitrary non-zero isotropic vector.

We consider the kernel
$$\calU_{\calV,x}:=\{u \in \calV : u(x)=0\}$$
of the surjective linear mapping
$$u \in \calV \mapsto u(x) \in \calV x.$$
Any $u \in \calU_{\calV,x}$ stabilizes $\{x\}^\bot$ (because it stabilizes $\F x$)
and hence induces a nilpotent endomorphism $\overline{u}$ of the quotient space $\{x\}^\bot/\F x$.
Note that $b$ induces a non-degenerate symmetric or alternating bilinear form $\overline{b}$ on $\{x\}^\bot/\F x$
with Witt index $\nu-1$. The set
$$\calV \modu x:=\{\overline{u} \mid u \in \calU_{\calV,x}\}$$
is then a nilpotent linear subspace of $\calS_{\overline{b}}$ (respectively, of $\calA_{\overline{b}}$).
Finally, the kernel of the linear mapping
$$\varphi : u \in \calU_{\calV,x} \mapsto \overline{u} \in \calV \modu x$$
consists of the operators $u \in \calV$ that vanish at $x$ and map $\{x\}^\bot$ into $\F x$.
Set
$$L_{\calV,x}:=\bigl\{y \in V : \; x \otimes_b y \in \calV\bigr\}$$
(respectively,
$$L_{\calV,x}:=\bigl\{y \in V : \; x \wedge_b y \in \calV\bigr\}).$$
By Proposition \ref{caractensors}, $L_{\calV,x}$ is a linear subspace of $\{x\}^\bot$
(respectively, a linear subspace of $\{x\}^\bot$ that contains $x$),
the kernel of $\varphi$ reads $x \otimes_b L_{\calV,x}$ (respectively,
$x \wedge_b L_{\calV,x}$), and the dimension of that kernel
equals $\dim L_{\calV,x}$ (respectively, $\dim L_{\calV,x}-1$).
Hence, applying the rank theorem twice yields
$$\dim \calV=\dim (\calV x)+\dim L_{\calV,x}+\dim (\calV \modu x)$$
(respectively,
$$\dim \calV=\dim (\calV x)+\dim L_{\calV,x}-1+\dim (\calV \modu x)).$$
In any case, the Orthogonality Lemma for Tensors (Proposition \ref{ortholemmatensors})
yields that $\F x \oplus \calV x$ is $b$-orthogonal to $L_{\calV,x}$.
Hence,
$$\dim (\calV x)+\dim (L_{\calV,x}) \leq n-1,$$
and it follows that
$$\dim \calV\leq \dim (\calV \modu x)+n-1$$
(respectively,
$$\dim \calV\leq \dim (\calV \modu x)+n-2).$$
If we do not take Theorem \ref{majoTheo} for granted, this allows one to prove it by induction on the dimension of $V$:
indeed, by induction $\dim (\calV \modu x) \leq (\nu-1)\bigl(n-2-(\nu-1)\bigr)$
(respectively, $\dim (\calV \modu x) \leq (\nu-1)\bigl(n-2-(\nu-1)-1\bigr)$), and it follows that
$\dim \calV \leq \nu(n-\nu)$ (respectively, $\dim \calV \leq \nu(n-\nu-1)$).

Now, assume that
$\dim \calV=\nu(n-\nu)$ (respectively, $\dim \calV=\nu(n-\nu-1)$).
Then, the above series of inequalities yields the two equalities
$$\dim (\F x \oplus \calV x)+\dim L_{\calV,x}=n$$
and
$$\dim (\calV \modu x)=(\nu-1)\bigl(n-2-(\nu-1)\bigr)$$
(respectively,
$$\dim (\calV \modu x)=(\nu-1)\bigl(n-2-(\nu-1)-1\bigr)).$$
Since $\F x \oplus \calV x$ is $b$-orthogonal to $L_{\calV,x}$, those two equalities yield
the following important results:

\begin{lemma}[Strong Orthogonality Lemma]\label{strongorthoLemma}
Let $x \in V$ be a non-zero $b$-isotropic vector, and let $\calV$
be a nilpotent subspace of $\calS_b$ (respectively, of $\calA_b$) with dimension
$\nu(n-\nu)$ (respectively, $\nu(n-\nu-1)$).

Then, $\F x \oplus \calV x$ is the $b$-orthogonal complement of $L_{\calV,x}$.
\end{lemma}

\begin{lemma}[Preservation of Maximality Lemma]\label{preservemaxLemma}
Let $x \in V$ be a non-zero $b$-isotropic vector, and let $\calV$
be a nilpotent subspace of $\calS_b$ (respectively, of $\calA_b$) with dimension
$\nu(n-\nu)$ (respectively, $\nu(n-\nu-1)$). Denote by $\overline{b}$
the bilinear form on $\{x\}^\bot/\F x$ induced by $b$.

Then, $\calV \modu x$ is a nilpotent subspace of $\calS_{\overline{b}}$ (respectively,
of $\calA_{\overline{b}}$) with the maximal dimension among such subspaces.
\end{lemma}

Related to the above method is the following basic result:

\begin{lemma}[Lifting Lemma]\label{liftinglemma}
Let $x \in V$ be a non-zero $b$-isotropic vector, and let $\calV$
be a nilpotent subspace of $\calS_b$ (respectively, of $\calA_b$) with dimension
$\nu(n-\nu)$ (respectively, $\nu(n-\nu-1)$).
Denote by $\overline{b}$
the bilinear form on $\{x\}^\bot/\F x$ induced by $b$.

Assume that $\calV x=\{0\}$ and that $\calV \modu x$ equals $\WS_{\overline{b},\calG}$ (respectively, $\WA_{\overline{b},\calG}$) for some
maximal partially complete $\overline{b}$-singular flag $\calG$ of $\{x\}^\bot/\F x$.

Then, $\calV=\WS_{b,\calF}$ (respectively, $\calV=\WA_{b,\calF}$) for some
maximal partially complete $b$-singular flag $\calF$ of $V$.
\end{lemma}

\begin{proof}
Write $\calG=(G_i)_{0 \leq i \leq \nu-1}$, set $F_0:=\{0\}$ and, for
$k \in \lcro 1,\nu\rcro$, denote by $F_k$ the inverse image of $G_{k-1}$ under the
canonical projection of $\{x\}^\bot$ onto $\{x\}^\bot/\F x$. Then,
one sees that $(F_0,\dots,F_\nu)$ is a maximal partially complete $b$-singular flag of $V$, and
every element of $\calV$ stabilizes it (here $\calU_{\calV,x}=\calV$ since we have assumed that $\calV x=\{0\}$).
Hence, $\calV \subset \WS_{b,\calF}$ (respectively $\calV \subset \WA_{b,\calF}$)
and since the dimensions of those spaces are equal we conclude that $\calV=\WS_{b,\calF}$
(respectively, $\calV=\WA_{b,\calF}$).
\end{proof}

\subsection{Stability properties of spaces with the maximal dimension}\label{stabcubereview}

In the course of \cite{dSPStructured1}, we have obtained an additional structural result
on nilpotent subspaces of $\calS_b$ or $\calA_b$ with the maximal possible dimension:

\begin{lemma}\label{cubeslemma}
Let $b$ be a non-degenerate symmetric or alternating bilinear form on $V$, with Witt index $\nu$.
Let $\calV$ be a nilpotent linear subspace of $\calS_b$ (respectively, of $\calA_b$)
with dimension $\nu(n-\nu)$ (respectively, $\nu(n-\nu-1)$).
Assume furthermore that $|\F| >3$.
Then, $u^3 \in \calV$ for all $u \in \calV$.
\end{lemma}

See lemmas 5.2 and 5.3 of \cite{dSPStructured1}. We use a simple linearity trick to derive the following
corollary, which will turn out to be very useful when $b$ is symmetric:

\begin{cor}\label{cubescor}
Let $b$ be a non-degenerate symmetric or alternating bilinear form on $V$, with Witt index $\nu$.
Let $\calV$ be a nilpotent linear subspace of $\calS_b$ (respectively, of $\calA_b$)
with dimension $\nu(n-\nu)$ (respectively, $\nu(n-\nu-1)$).
Assume furthermore that $|\F| >3$.

Then, $u^2v+uvu+vu^2 \in \calV$ for all $u$ and $v$ in $\calV$.
\end{cor}

\begin{proof}
Let $u$ and $v$ belong to $\calV$.
By Lemma \ref{cubeslemma}, the linear subspace $\calV$ contains
$$(u+v)^3-(u-v)^3-2v^3=2(u^2v+uvu+vu^2)$$
and we conclude by using the fact that the characteristic of $\F$ is not $2$.
\end{proof}

\subsection{Extending scalars}\label{extensionscalarsSection}

Let $\calV$ be a nilpotent subspace of $\calS_b$ (respectively, of $\calA_b)$.
Let $\L$ be a field extension of $\F$. We assume that the extension $\F - \L$
preserves Witt indices, i.e. it induces an injection from the Witt group of $\F$
to the one of $\L$. In particular, this is known to hold if $\L$ is purely transcendental over $\F$,
e.g.\ when $\L=\F(t)$ (see \cite{invitquad} Chapter XV, Section 3). This case is interesting because $\L$ then turns out to be infinite.
Next, we set $V_\L:=\L \otimes_\F V$. Given $u \in \End(V)$, we define $u_\L$ as the endomorphism of the $\L$-vector space
$V_\L$ such that $u_\L(1 \otimes x)=u(x)$ for all $x \in V$.
The mapping $u \in \End(V) \mapsto u_\L \in \End_\L(V_\L)$ is a homomorphism of $\F$-algebras.
Given a linear subspace $\calV$ of $\End(V)$, it is folklore that
the span $\calV_\L$ of $\{u_\L \mid u \in \calV\}$ in the $\L$-vector space $\End_\L(V_\L)$
has its dimension over $\L$ equal to the one of $\calV$ over $\F$; better still, for every basis
$(u_1,\dots,u_k)$ of $\calV$, the family $\bigl((u_1)_\L,\dots,(u_k)_\L\bigr)$ is a basis of $\calV_\L$.

Next, let $\calV$ be a nilpotent subspace of $\calS_b$ (respectively, of $\calA_b$), and denote by $d$
the generic nilindex of $\calS_b$ (respectively, of $\calA_b$).
Assume that $|\F| \geq d$. We claim that $\calV_\L$ is nilpotent.
Indeed, let us take a basis $(u_1,\dots,u_p)$ of $\calV$.
Then,
$$\forall (\lambda_1,\dots,\lambda_p)\in \F^p, \; \biggl(\sum_{k=1}^p \lambda_k\, u_k\biggr)^d=0,$$
which, by extending to $V_\L$, yields
$$\forall (\lambda_1,\dots,\lambda_p)\in \F^p, \; \biggl(\sum_{k=1}^p \lambda_k\, (u_k)_{\L}\biggr)^d=0.$$
On the left-hand side is a (vector-valued) polynomial function, homogenous with degree $d$;
since we have assumed that $|\F|\geq d$, this function vanishes everywhere on $\L^p$, whence
every element of $\calV_\L$ is nilpotent.

Next, $b$ induces a bilinear form $b_\L$ on the $\L$-vector space $V_\L$ such that
$$\forall (x,y)\in V^2, \quad b_\L(1 \otimes x,1 \otimes y)=b(x,y).$$
This form is symmetric if $b$ is symmetric, alternating if $b$ is alternating.
Finally, it is clear that $u_\L$ is $b_\L$-symmetric (respectively, $b_\L$-alternating) for all $u \in \calS_b$ (respectively, $u \in \calA_b$).
Hence, $\calV_\L$ is a linear subspace of $\calS_{b_\L}$ (respectively, of $\calA_{b_\L}$).

Finally, assume that there is a non-zero isotropic vector $y$ of $V_\L$ at which every element of $\calV_\L$
vanishes. There is a unique minimal subspace $E$ of $V$ such that $y$ can be written as a sum of tensors
of the form $\lambda \otimes x$ with $\lambda \in \L$ and $x \in E$.
Then, every element of $\calV$ must vanish at every element of $E$.
Moreover, we claim that $E$ contains a non-zero $b$-isotropic vector. Indeed, let
$(e_1,\dots,e_q)$ be a basis of $E$. We write $y=\underset{k=1}{\overset{q}{\sum}} \lambda_k \otimes e_k$
for some $(\lambda_k)_{1 \leq k \leq q} \in \L^q$. Since $b_\L(y,y)=0$, the quadratic form
$$Q : (\mu_1,\dots,\mu_q)\in \F^q \mapsto b\biggl(\sum_{k=1}^q \mu_k e_k,\sum_{k=1}^q \mu_k e_k\biggr)=\sum_{1 \leq k,l \leq q} \mu_k\, \mu_l\, b(e_k,e_l)$$
becomes isotropic over $\L$.
Since we have assumed that the extension $\L$ of $\F$ preserves Witt indices, we conclude
that $Q$ is isotropic, which yields a non-zero $b$-isotropic vector $x$ of $E$.
We conclude that every element of $\calV$ vanishes at~$x$.

\vskip 3mm
Combining the above remarks with Corollary \ref{nilindexCor}, we obtain the following results:

\begin{prop}\label{extensionprop}
Let $\L$ be a field extension of $\F$ that preserves Witt indices. Let $b$ be a non-degenerate symmetric or alternating bilinear form
on an $\F$-vector space $V$ with finite dimension $n$. Denote by $\nu$ the Witt index of $b$.
Let $\calV$ be a nilpotent linear subspace of $\calS_b$ or of $\calA_b$.
\begin{enumerate}[(a)]
\item If $b$ is alternating, $|\F| \geq n$ and $\calV \subset \calS_b$, then
$\calV_\L$ is a nilpotent subspace of~$\calS_{b_\L}$.
\item If $b$ is symmetric, $n \neq 2\nu$, $|\F| \geq 2\nu+1$ and $\calV\subset \calA_b$,
then $\calV_\L$ is a nilpotent subspace of~$\calA_{b_\L}$.
\item If $b$ is symmetric, $n=2\nu$, $|\F| \geq n-1$ and $\calV \subset \calA_b$,
then $\calV_\L$ is a nilpotent subspace of~$\calA_{b_\L}$.
\item If there exists a non-zero $b_\L$-isotropic vector of $V_\L$ at which all the elements of $\calV_\L$
vanish, then there exists a non-zero $b$-isotropic vector of $V$ at which all the elements of $\calV$
vanish.
\end{enumerate}
\end{prop}

We are now ready to conclude.

\begin{prop}\label{reductiontoinfinitealt}
If Theorem \ref{alttheo} holds over any infinite field (with characteristic not $2$), then it holds over any field (with characteristic not $2$) that satisfes its cardinality requirements.
\end{prop}

\begin{proof}
Assume that Theorem \ref{alttheo} holds over any infinite field with characteristic not $2$.
Then, we prove Theorem \ref{alttheo} by induction on the dimension of the vector space $V$.
Let $\F$ be a field with characteristic not $2$, $V$ be an $n$-dimensional vector space over $\F$ (with $n>0$),
$b$ be a symmetric bilinear form on $V$, and
$\calV$ be a nilpotent subspace of $\calA_b$ with dimension $\nu(n-\nu-1)$.  Assume that $|\F| \geq 2\nu+1$ if $n>2\nu$, and that $|\F| \geq n-1$ otherwise. Let us choose an infinite extension $\L$ of $\F$ that preserves Witt indices (e.g.\ $\L=\F(t)$).
By Proposition \ref{extensionprop}, $\calV_\L$ is a nilpotent subspace of $\calA_{b_\L}$
with dimension $\nu(n-\nu-1)$. Since $\L$ is infinite we deduce that
$\calV_\L=\WA_{b_\L,\calF}$ for some maximal partially complete $b_\L$-singular flag $\calF$ of $V_\L$.
In particular, there is a non-zero $b_\L$-isotropic vector of $V_\L$ at which all the elements of $\calV_\L$
vanish. By point (d) of Proposition \ref{extensionprop}, there exists a non-zero isotropic vector $x$ of $V$
at which all the elements of $\calV$ vanish.

Denote by $\overline{b}$ the symmetric bilinear form
induced by $b$ on $\{x\}^\bot/\F x$.
By Lemma \ref{preservemaxLemma}, $\calV \modu x$ is a nilpotent subspace of $\calA_{\overline{b}}$ with dimension $(\nu-1)\bigl((n-2)-(\nu-1)-1\bigr)$.
Noting that $|\F|\geq 2\nu+1>2\nu-1$ if $2(\nu-1)\neq n-2$ (i.e.\ if $2\nu \neq n$), and $|\F| \geq n-1 \geq n-3$ otherwise, we
find by induction that $\calV \modu x=\WA_{\overline{b},\calG}$ for some
maximal partially complete $\overline{b}$-singular flag $\calG$ of $\{x\}^\bot/\F x$.
Then, the Lifting Lemma (Lemma \ref{liftinglemma}) yields that $\calV=\WA_{b,\calF}$ for some
maximal partially complete $b$-singular flag $\calF$ of $V$.
\end{proof}

With the same line of reasoning, we also get:

\begin{prop}\label{reductiontoinfinitesym}
If Theorem \ref{symtheo} holds over any infinite field (with characteristic not $2$), then it holds over any field (with characteristic not $2$) that satisfies its cardinality requirements.
\end{prop}

\section{Spaces of symmetric nilpotent endomorphisms for a symplectic form}\label{symproofsection}

We are now ready for the proofs of Theorems \ref{alttheo} and \ref{symtheo}.
In the present section, we deal with the latter, which turns out to be the easier.

By Proposition \ref{reductiontoinfinitesym}, Theorem \ref{symtheo} requires a proof only for infinite fields.
So, we fix an infinite field $\F$ (with characteristic not $2$) and we prove Theorem \ref{symtheo} by induction on the dimension of the vector space $V$ under consideration. The result is obvious for a vector space with dimension $0$.
Let $V$ be vector space over $\F$ with finite dimension $n>0$, and let $b$ be a symplectic form on $V$.
Let $\calV$ be a nilpotent subspace of $\calS_b$ with dimension $\nu(n-\nu)$, where $\nu:=\frac{n}{2}\cdot$
By induction and the Lifting Lemma (Lemma \ref{liftinglemma}), it suffices to show that there exists
a non-zero vector $x \in V$ such that $\calV x=\{0\}$. We split the proof of the existence of such a vector
into two subcases, whether the generic nilindex of $\calV$ is at most $2$ or not.

\subsection{A special case}\label{nilindex2section}

Here, we assume that all the elements of $\calV$ have square zero.
Hence $uv+vu=(u+v)^2-u^2-v^2=0$ for all $(u,v)\in \calV^2$.
Classically, any subset $\calA$ of skew-commuting nilpotent elements of $\End(V)$ vanishes at some non-zero vector. To see this, one can use Jacobson's triangularization theorem \cite{Jacobson,Radjavi} for example, but one can also give a very elementary proof: if $\calA \subset \{0\}$ the result is obvious; otherwise choose $u \in \calA \setminus \{0\}$ and note that $\Ker u$ is stable under all the elements of $\calA$. By restricting the elements of $\calA$ to $\Ker u$, we recover a skew-commuting set of square-zero endomorphisms of $\Ker u$, and by induction there is a non-zero vector $x$ of $\Ker u$ at which all the elements of $\calA$ vanish.

\subsection{When the generic nilindex is greater than $2$ (part 1)}

In the rest of the proof, we assume that the generic nilindex of $\calV$ is greater than $2$.

\begin{claim}
The generic nilindex of $\calV$ is greater than or equal to $n-2$.
\end{claim}

\begin{proof}
Choose an arbitrary vector $x \in V \setminus \{0\}$.
Denote by $\overline{b}$ the symplectic form induced by $b$ on $\{x\}^\bot/\F x$.
By Lemma \ref{preservemaxLemma}, the subspace $\calV \modu x$ has dimension $(\nu-1)\bigl((n-2)-(\nu-1)\bigr)$
and it is a nilpotent subspace of $\calS_{\overline{b}}$. Hence, by induction $\calV \modu x=\WS_{\overline{b},\calG}$
for some maximal partially complete $\overline{b}$-singular flag $\calG$ of $\{x\}^\bot/\F x$.
Hence, by Lemma \ref{genericnilindexsym} there is an element $v \in \calV \modu x$ with nilindex $n-2$.
We conclude that $\ind(u) \geq n-2$ for some $u \in \calV$.
\end{proof}

In the rest of the proof, we denote by $p$ the generic nilindex of $\calV$.
Remember that we have assumed $p \geq 3$.

\begin{claim}
The nilpotent subspace $\calV$ is pure.
\end{claim}

\begin{proof}
Let $u \in \calV$ have nilindex $p$.
Since $p \geq n-2$ and $p \geq 3$, we have $2p>n$, and hence $u$ has exactly one Jordan cell of size $p$.
\end{proof}

\begin{claim}\label{symmetrictensorclaim}
For all $x \in \calV^\bullet \setminus \{0\}$, the space $\calV$ contains $x \otimes_b x$.
\end{claim}

\begin{proof}
Let $x \in \calV^\bullet \setminus \{0\}$ and choose $u \in \calV$ such that $x \in \im u^{p-1}$.
By Proposition \ref{tangentprop}, $\calV x \subset \im u$, whence $\F x \oplus \calV x \subset \im u$.
Taking the orthogonal complement and applying the Strong Orthogonality Lemma (Lemma \ref{strongorthoLemma}),
we deduce that $\Ker u=(\im u)^\bot \subset L_{\calV,x}$, and in particular $x \otimes_b x \in \calV$
since $x \in \Ker u$.
\end{proof}

Now, we use a \emph{reductio ad absurdum}, by assuming that $\calV x \neq \{0\}$ for all $x \in \calV^\bullet$.
The Reduction Lemma yields
$$\text{(H) :} \qquad \forall x \in \calV^\bullet \setminus \{0\}, \; K(\calV) \not\subset \F x \oplus \calV x.$$

\begin{claim}
The space $K(\calV)$ is totally $b$-singular.
\end{claim}

\begin{proof}
Assume otherwise. Then, $K(\calV) \cap K(\calV)^\bot$ is a proper linear subspace of $K(\calV)$.
Hence, the Linear Density Lemma (Lemma \ref{LinearDensityLemma}) yields a basis
$(e_1,\dots,e_q)$ of $K(\calV)$ made of vectors of $\calV^\bullet \setminus K(\calV)^\bot$.
Applying the Linear Density Lemma once more, we find an additional vector $x \in \calV^\bullet \setminus \{0\}$
such that $b(x,e_i) \neq 0$ for all $i \in \lcro 1,q\rcro$.

Now, let $i \in \lcro 1,q\rcro$. Then, $(e_i \otimes_b e_i)(x)=2\,b(e_i,x)\,e_i$ is non-zero
and collinear with $e_i$. Thus, $e_i \in \calV x$ by Claim \ref{symmetrictensorclaim}. We deduce that $K(\calV) \subset \calV x$,
which contradicts property (H).
\end{proof}

\subsection{When the generic nilindex is greater than $2$ (part 2)}\label{endaltsym}

Now, we are ready for the final contradiction.
Choose a basis $(x_1,\dots,x_q)$ of $K(\calV)$ made of vectors of $\calV^\bullet$.

Let $i \in \lcro 1,q\rcro$. By the Strong Orthogonality Lemma, we use the non-inclusion
$K(\calV) \not\subset \F x_i \oplus \calV x_i$ to obtain a vector
$z_i \in V$ such that $x_i \otimes_b z_i \in \calV$ and $z_i \not\in K(\calV)^\bot$.
Applying the Linear Density Lemma once more, we find a vector $y \in \calV^\bullet \setminus \{0\}$ such that
$b(z_i,y) \neq 0$ for all $i \in \lcro 1,q\rcro$.
Hence, for all $i \in \lcro 1,q\rcro$, we have $(x_i \otimes_b z_i)(y)=b(z_i,y)\, x_i$ (because $b(x_i,y)=0$ by the total singularity of $K(\calV)$)
and we deduce that $x_i \in \calV y$. Hence $y$ satisfies condition (C) of the Reducibility Lemma, contradicting property (H).
This final contradiction completes the proof.

\section{Spaces of nilpotent alternating endomorphisms for a symmetric bilinear form}\label{altproofsection}

In this ultimate section, we prove Theorem \ref{alttheo}.
By Proposition \ref{reductiontoinfinitealt}, we know that it suffices to do so when the underlying field
is infinite. We fix such a field $\F$ (with characteristic not $2$), and we prove the result by induction on the dimension of the vector space $V$ under consideration. Let $V$ be a finite-dimensional vector space with dimension $n>0$, and let $b$ be a non-degenerate symmetric bilinear form on $V$, whose Witt index we denote by $\nu$.
Let $\calV$ be a nilpotent subspace of $\calA_b$ with dimension $\nu(n-\nu-1)$. If $\nu=0$ then $\calV=\{0\}$, which is the claimed result.
Assume now that $\nu>0$. Combining the Lifting Lemma and the induction hypothesis, we see that it suffices to prove the existence of
a non-zero $b$-isotropic vector $x$ of $V$ such that $\calV x=\{0\}$.
The overall strategy of the proof is similar to the one of the previous section, but the details
are much more difficult.

Assume first that the generic nilindex of $\calV$ equals $2$. Then, as in Section \ref{nilindex2section}, we find that the elements
of $\calV$ skew-commute pairwise. Moreover, we can take $u_0 \in \calV$ with nilindex $2$,
so that $\im u_0$ is a non-trivial subspace of $V$, and it is totally $b$-singular because
$\im u_0\subset \Ker u_0=(\im u_0)^\bot$. The elements of $\calV$ induce nilpotent endomorphisms of $\im u_0$
that skew-commute pairwise, and with the same line of reasoning as in Section \ref{nilindex2section} we deduce that some
non-zero vector $x$ of $\im u_0$ satisfies $\calV x=\{0\}$. The vector $x$ is $b$-isotropic because
it belongs to $\im u_0$, and hence the proof is complete in that case.

In the remainder of the proof, we assume that the generic nilindex of $\calV$ is greater than $2$.
In the next section, we obtain additional information on the generic nilindex of $\calV$.
We finish the present one with a basic, yet very useful result:

\begin{claim}
Every vector of $\calV^\bullet$ is isotropic.
\end{claim}

\begin{proof}
Denote by $p$ the generic nilindex of $\calV$. Let $x \in \calV^\bullet$. Then, for some $u \in \calV$, we have
$x \in \im u^{p-1}$, and hence $x$ belongs both to $\Ker u$ and to $\im u=(\Ker u)^\bot$. This yields the claimed result.
\end{proof}

\subsection{The generic nilindex}

\begin{claim}\label{minindex}
The generic nilindex of $\calV$ is at least $2\nu-1$ if $n\neq 2\nu$, and at least $n-3$ otherwise.
\end{claim}

\begin{proof}
Choose an arbitrary isotropic vector $x \in V \setminus \{0\}$.
Denote by $\overline{b}$ the non-degenerate symmetric bilinear form induced by $b$ on $\{x\}^\bot/\F x$. Its Witt index equals $\nu-1$.
By Lemma \ref{preservemaxLemma}, the subspace $\calV \modu x$ has dimension $(\nu-1)\bigl((n-2)-(\nu-1)-1\bigr)$
and it is a nilpotent subspace of $\calA_{\overline{b}}$. Hence, by induction $\calV \modu x=\WA_{\overline{b},\calG}$
for some maximal partially complete $\overline{b}$-singular flag $\calG$ of $\{x\}^\bot/\F x$.
By Lemma \ref{genericnilindexalt} there is an element $v \in \calV \modu x$ with nilindex $n-3$ if $n-2=2(\nu-1)$,
and with nilindex $2\nu-1$ otherwise. This proves the claimed statement.
\end{proof}

From now on, the generic nilindex of $\calV$ is denoted by $p$.

\begin{claim}\label{Jordanstruct1}
Assume that $\calV$ is pure, and let $u \in \calV$ be with nilindex $p$. Then, $p$ is odd and:
\begin{itemize}
\item Either $u$ has one Jordan cell of size $p$, one Jordan cell of size $3$, and all its other Jordan cells have size $1$.
\item Or $u$ has one Jordan cell of size $p$, and all its other Jordan cells have size~$1$.
\end{itemize}
\end{claim}

\begin{proof}
By Lemma \ref{Jordansymalt}, the number of Jordan cells of given even size for $u$
is even, and since $\calV$ is pure we get that $p$ is odd.

Set $\varepsilon:=1$ if $n=2\nu$, otherwise set $\varepsilon:=0$.
The combination of Corollary \ref{MaxRankNil} and of Claim \ref{minindex}
yields $2\nu-1-2\varepsilon \leq p \leq 2\nu+1-2\varepsilon$.
Moreover, $\rk u \leq 2\nu-2\epsilon$. This leaves us with three possibilities:
\begin{enumerate}[(i)]
\item $u$ has exactly one Jordan cell of size $p$, and all the other Jordan cells have size $1$;
\item $u$ has exactly one Jordan cell of size $p$, exactly one Jordan cell of size $3$, and all its other Jordan cells have size $1$;
\item $u$ has exactly one Jordan cell of size $p$, exactly two Jordan cells of size $2$, and all its other Jordan cells have size $1$.
\end{enumerate}
We need to discard case (iii). Assume on the contrary that it holds.
Then, $p+4 \leq n$, and hence $2\nu \leq p+3<n$, leading to $\epsilon=0$.
Besides, Proposition \ref{Jordansymalt} would yield
a decomposition $V=V_p \overset{\bot}{\bigoplus} V_2  \overset{\bot}{\bigoplus} V_1$ into a $b$-orthogonal direct sum
in which all the subspaces $V_p$, $V_2$ and $V_1$ are stable under $u$, the endomorphism of $V_p$ induced by $u$ is a Jordan cell of size $p$,
and the endomorphism of $V_2$ induced by $u$ is the direct sum of two Jordan cells of size $2$.
But then Proposition \ref{Jordansymalt} would also yield that the Witt index of $b$ is at least
$\frac{p-1}{2}+2 \geq \nu-1+2=\nu+1$, which is absurd.

Hence, only cases (i) and (ii) are possible.
\end{proof}

\begin{claim}\label{Jordanstruct2}
Assume that $\calV$ is not pure. Then,
$p=3$ and there is an element of $\calV$ that has exactly two Jordan cells of size $3$, and no Jordan cell of size $2$.
\end{claim}

\begin{proof}
Choose $u \in \calV$ with nilindex $p$ and at least $2$ Jordan cells of size $p$.
Set $\varepsilon:=1$ if $n=2\nu$, otherwise set $\varepsilon:=0$.
Again, we have $2\nu-1-2\varepsilon \leq p \leq 2\nu+1-2\varepsilon$
and $\rk u \leq 2\nu-2\epsilon$.
Let $u \in \calV$ have nilindex $p$ and several Jordan cells of size $p$:
denote by $k$ the number of those cells, and by $l$ the number of cells
of size less than $p$ and more than $1$. Then, $\rk u \geq (p-1)+(k-1)(p-1)+l$,
whence $(k-1)(p-1)+l \leq 2$. As $p \geq 3$ and $k-1>0$, this yields $p=3$, $k=2$ and $l=0$, which proves the claimed statement.
\end{proof}

\subsection{A variation of Proposition \ref{tangentprop}}

\begin{claim}\label{orthotangentclaim}
Let $u \in \calV$ have nilindex $p$, and let $x \in \im u^{p-1} \setminus \{0\}$
and $v \in \calV$. Then, $v(x) \in u(\{x\}^\bot)$.
\end{claim}

\begin{proof}
Choose $y \in V$ such that $u^{p-1}(y)=x$.
For $\lambda \in \F$, set
$$\gamma(\lambda):=(u+\lambda v)^{p-1}(y)=\sum_{k=0}^{p-1} \lambda^k x_k,$$
where $x_0,\dots,x_{p-1}$ all belong to $V$ and $x_0=x$.

We have seen in the proof of Proposition \ref{tangentprop} that
$v(x)=u(-x_1)$. For all $\lambda \in \F$, the vector $\gamma(\lambda)$ is $b$-isotropic because it belongs both to the kernel and to the range of the $b$-alternating endomorphism $u+\lambda v$; hence
$b(\gamma(\lambda),\gamma(\lambda))=0$.
Differentiating this polynomial function at zero yields $2\,b(x_0,x_1)=0$, and it follows that $-x_1 \in \{x\}^\bot$,
which yields the claimed statement.
\end{proof}

\subsection{Investigating the structure of $\F x \oplus \calV x$}

\begin{claim}\label{stabu2}
Let $u \in \calV$ and $x \in \im u^{p-1}$. Then, $\calV x$ is stable under $u^2$.
\end{claim}

\begin{proof}
Let $v \in \calV$. By Corollary \ref{cubescor}, the endomorphism $w:=u^2 v+uvu+vu^2$ belongs to $\calV$, and we see that $u^2(v(x))=w(x)$
because $x \in \Ker u$.
\end{proof}

\begin{claim}\label{wedgeclaim}
Let $x \in V$ be isotropic and non-zero. Assume that
$(\F x \oplus \calV x)^\bot$ contains a non-isotropic vector.
Then, for all $v \in \calV$, the subspace $\calV$ contains $x \wedge_b v(x)$.
\end{claim}

\begin{proof}
Let $v \in \calV$.
By assumption and by the Strong Orthogonality Lemma, there is a non-isotropic vector $y$ of $\{x\}^\bot$ such that
$x \wedge_b y \in \calV$.
Set $u:=x \wedge_b y$ and $\alpha:=b(y,y)$, the latter of which is non-zero.
One computes that $u^2=-\frac{\alpha}{2} x \otimes_b x$.
Another computation shows that
$$v \circ (x \otimes_b x)+(x \otimes_b x)\circ v=-2\,x \wedge_b v(x).$$
Finally,
$$uvu=b(y,v(x))\,u.$$
Applying Corollary \ref{cubescor}, we deduce that the
operator $\alpha\, x \wedge_b v(x)+b(y,v(x))\,u$ belongs to the linear subspace $\calV$.
Since $\alpha \neq 0$ and $u \in \calV$, the desired conclusion follows.
\end{proof}

\begin{claim}\label{claimpowers}
Let $x \in V$ be isotropic and non-zero. Assume that
$(\F x\oplus \calV x)^\bot$ contains a non-isotropic vector. Then,
$v^k(x)$ is orthogonal to $x$ for every $v \in \calV$ and every integer $k \geq 0$.
Moreover, $\F x \oplus \calV x$ is totally singular.
\end{claim}

\begin{proof}
Let $v \in \calV$.
The first result is already known for odd $k$ because $v$ is $b$-alternating; on the other hand it is known for $k=0$.
Now, let $k$ be a non-zero even integer. We write $b(x,v^k(x))=-b(v(x),v^{k-1}(x))$.
By Claim \ref{wedgeclaim}, $x \wedge_b v(x)$ belongs to $\calV$. By the Orthogonality Lemma for Tensors,
we deduce that $b(v(x),v^{k-1}(x))=0$ and the first claimed statement ensues.
It follows in particular that $b(x,v(x))=0$, $b(v(x),v(x))=-b(x,v^2(x))=0$ and $b(x,x)=0$ for all $v \in \calV$, which
yields that $\F x \oplus \calV x$ is totally singular.
\end{proof}

From there, we split the study into two subcases, whether $\calV$ is pure or not.

\subsection{Case 1: $\calV$ is pure}

Here, we assume that $\calV$ is pure.

\begin{claim}\label{halfdimclaim}
Let $u \in \calV$ have nilindex $p$, and let $x \in \im u^{p-1} \setminus \{0\}$.
Then, $\dim(\F x \oplus \calV x) \leq \frac{n}{2}$, and the inequality is strict if
$\F x\oplus \calV x$ is totally singular.
\end{claim}

\begin{proof}
\textbf{Case 1:} $u$ has one Jordan cell of size $p$ and all its other Jordan cells have size $1$. \\
Hence $\F x=\im u^{p-1}$, so that $\{x\}^\bot=\Ker u^{p-1}$.
Then, there is a basis $\bfB=(e_1,\dots,e_p,f_1,\dots,f_{n-p})$ of $V$ in which
$u(e_i)=e_{i-1}$ for all $i \in \lcro 2,p\rcro$, and $u(e_1)=u(f_1)=\cdots=u(f_{n-p})=0$. We note that $\im u^{p-1}=\F x$ and
$$\{x\}^\bot=\Ker u^{p-1}=\Vect(e_1,\dots,e_{p-1},f_1,\dots,f_{n-p}).$$
By Claim \ref{orthotangentclaim}, we have $\calV x \subset u(\{x\}^\bot)=\Vect(e_1,\dots,e_{p-2})$.
For $i \in \lcro 0,p\rcro$, set $F_i:=\Vect(e_1,\dots,e_i)$.
We claim that $\calV x \cap F_{2i+1}=\calV x \cap F_{2i}$ for all $i \in \lcro 0,(p-1)/2\rcro$.
Indeed, assuming otherwise, there would be an index $i \in \lcro 0,(p-1)/2\rcro$ and
 a vector $y$ of $\calV x  \cap F_{2i+1}$ with a non-zero coefficient $\alpha$ on $e_{2i+1}$
in the basis $\bfB$: applying $(u^2)^i$ would yield that $\alpha e_1 \in \calV x$ (see Claim \ref{stabu2}), contradicting the fact that $\calV x$ is linearly disjoint from $\F x$. Then, as $\dim (\calV x \cap F_{2i}) \leq \dim (\calV x \cap F_{2i-1})+1$
for all $i \in \lcro 1,(p-1)/2\rcro$, we obtain by induction that
$\dim (\calV x \cap F_{2i+1}) \leq i$ for all $i \in \lcro 0,(p-1)/2\rcro$. In particular,
$$\dim \calV x=\dim (\calV x \cap F_{p-2}) \leq \frac{p-3}{2},$$
and hence
$$\dim(\F x\oplus \calV x) \leq \frac{p-1}{2}<\frac{n}{2}\cdot$$

\vskip 2mm
\noindent \textbf{Case 2:} $p>3$ and $u$ has a Jordan cell of size $3$. \\
Then, by Claim \ref{Jordanstruct1} together with Proposition \ref{Jordansymalt}, we have
a decomposition $V=V_0 \overset{\bot}{\oplus} V_1 \overset{\bot}{\oplus} V_2$
into a $b$-orthogonal direct sum, in which all subspaces $V_0,V_1,V_2$ are stable under $u$,
and the resulting endomorphisms are, respectively, a Jordan cell of size $p$, a Jordan cell of size $3$, and the zero endomorphism of $V_2$.

We take a Jordan basis $(e_1,\dots,e_p)$ for $u_{V_0}$, a Jordan basis $(f_1,f_2,f_3)$ for $u_{V_1}$
and a basis $(g_1,\dots,g_{n-p-3})$ of $V_2$. Hence, $\bfB:=(e_1,\dots,e_p,f_1,f_2,f_3,g_1,\dots,g_{n-p-3})$ is a basis of $V$.
Again, $\im u^{p-1}=\F x=\F e_1$, whence
$$\{x\}^\bot=\Ker u^{p-1}=\Vect(e_1,\dots,e_{p-1},f_1,f_2,f_3,g_1,\dots,g_{n-p-3}).$$
Set $F_i:=\Vect(e_1,\dots,e_i)$ for all $i \in \lcro 0,p\rcro$.
We deduce from
Claim \ref{orthotangentclaim} that $\F x \oplus \calV x \subset F_{p-2} \oplus \Vect(f_1,f_2)$.
Combining this with Claim \ref{stabu2}, we find $u^2(\F x \oplus \calV x) \subset \calV x \cap F_{p-4}$.
Besides, with the same line of reasoning as in the first case above, we find
$$\dim(\calV x \cap F_{p-4}) \leq \frac{p-5}{2}\cdot$$
Noting that $\Ker u^2 \cap (\F x \oplus \calV x) \subset \Vect(e_1,e_2,f_1,f_2)$, the rank theorem yields
$$\dim(\F x \oplus \calV x) \leq \dim(\calV x \cap F_{p-4})+4 \leq \frac{p+3}{2} \leq \frac{n}{2},$$
and if equality holds then $\F x \oplus \calV x$ includes $\Vect(e_1,e_2,f_1,f_2)$, and in particular
it contains the vector $f_2$, which is non-isotropic by Lemma \ref{JordancellLemma}.
\end{proof}

\begin{claim}\label{claimexistnonisotropic}
Let $u \in \calV$ have nilindex $p$, and let $x \in \im u^{p-1} \setminus \{0\}$.
Then, $(\F x\oplus \calV x)^\bot$ contains a non-isotropic vector.
\end{claim}

\begin{proof}
Assume on the contrary that $G:=(\F x \oplus \calV x)^\bot$ is totally singular.
The space $G$ has dimension at least $\frac{n}{2}$, due to Claim \ref{halfdimclaim}. Therefore, $G^\bot=G$ because
$G \subset G^\bot$. Then, $\F x \oplus \calV x=G^\bot=G$ would be totally singular.
Using Claim \ref{halfdimclaim} once more, it would follow that $\dim G < \frac{n}{2}$, contradicting what we have just found.
\end{proof}

\begin{claim}
The subspace $K(\calV)$ is totally singular.
\end{claim}

\begin{proof}
Let $x$ and $y$ belong to $\calV^\bullet \setminus \{0\}$. Assume that $b(x,y) \neq 0$.
Since $\calV$ is pure, we can choose $u \in \calV$ and $v \in \calV$ such that $\im u^{p-1}=\F x$ and $\im v^{p-1}=\F y$.
Thus, $x \not\in (\im v^{p-1})^\bot=\Ker v^{p-1}$, and hence
$v^{p-1}(x) \neq 0$. It follows that $y$ is collinear with $v^{p-1}(x)$.
Combining Claims \ref{claimpowers} and \ref{claimexistnonisotropic}, we find that
$x$ is orthogonal to $v^{p-1}(x)$, which contradicts the fact that $x$ is non-orthogonal to $y$.
We conclude that $x$ is orthogonal to $y$. Varying $x$ and $y$ yields the claimed result.
\end{proof}

The remainder of the proof is similar to the one in Section \ref{endaltsym}: the only formal difference is that
symmetric tensors must be replaced with alternating tensors. One obtains that
$\calV x=\{0\}$ for some $x \in \calV^\bullet \setminus \{0\}$, which completes the proof.

\subsection{Case 2: $\calV$ is not pure}

Here, we assume that $\calV$ is not pure.
By Claim \ref{Jordanstruct2}, the generic nilindex in $\calV$ is $3$ and there
is an element $u$ of $\calV$ that has exactly two Jordan cells of size $3$, with all remaining Jordan cells of size $1$.
By Proposition \ref{Jordansymalt}, $V$ splits into a $b$-orthogonal direct sum $V=V_1 \overset{\bot}{\oplus} V_2 \overset{\bot}{\oplus} V_3$
in which $u$ stabilizes $V_1$ and $V_2$ and induces Jordan cells of size $3$ on those spaces, and $V_3 \subset \Ker u$.
We choose respective Jordan bases $(e_1,e_2,e_3)$ and $(f_1,f_2,f_3)$ of $V_1$ and $V_2$ for $u$.

\begin{claim}
Every element of $\calV$ stabilizes $\im u^2$.
\end{claim}

\begin{proof}
Note that $\im u^2=\Vect(e_1,f_1)$.

We shall prove that $\calV e_1 \subset \Vect(e_1,f_1)$.

Noting that $\{e_1\}^\bot=\Vect(e_1,e_2) \oplus V_2 \oplus V_3$, we deduce from Claim \ref{orthotangentclaim}
that $\F e_1 \oplus \calV e_1 \subset \Vect(e_1,f_1,f_2)$. Hence, $\dim(\F e_1 \oplus \calV e_1) \leq 3 \leq \frac{n}{2}\cdot$

Now, note that the isotropic vectors of $\Vect(e_1,f_1,f_2)$ are exactly the vectors of $\Vect(e_1,f_1)$.
Indeed, by Lemma \ref{JordancellLemma} the vector $f_2$ is non-isotropic, whereas
$b(e_1,e_1)=b(f_1,f_1)=b(e_1,f_1)=b(e_1,f_2)=b(f_1,f_2)=0$ since $e_1,f_1$ belong to $\Ker u$ and $e_1,f_1,f_2$ belong to $\im u$.

In particular, if $\dim(\F e_1 \oplus \calV e_1)=\frac{n}{2}$, then $\F e_1 \oplus \calV e_1$ contains the non-isotropic vector $f_2$.
Thus, with the same proof as for Claim \ref{claimexistnonisotropic}, we deduce that
$(\F e_1 \oplus \calV e_1)^\bot$ contains a non-isotropic vector.
It follows from  Claim \ref{claimpowers} that $\F e_1 \oplus \calV e_1$ is totally $b$-singular, and we deduce from the above that
$\calV e_1 \subset \Vect(e_1,f_1)=\im u^2$.

Symmetrically $\calV f_1 \subset \im u^2$ and we conclude that every element of $\calV$ stabilizes $\im u^2$.
\end{proof}

From here, it is easy to conclude. The $2$-dimensional totally singular subspace $\im u^2$
is stable under all the elements of $\calV$. The induced endomorphisms are nilpotent endomorphisms of $\im u^2$,
and by Gerstenhaber's theorem in dimension $2$
(or trivially if the resulting space of endomorphisms contains only the zero element) there is a non-zero vector $x$ of $\im u^2$ that is annihilated by all those endomorphisms. Hence, $\calV x=\{0\}$ and $x$ is isotropic: the proof is complete.

\end{document}